\ifx \labelsONmargin\undefined

  \ifx\RequirePackage\undefined
    \expandafter\ifx\csname amsart.sty\endcsname\relax
        \documentstyle[12pt,amssymb,amscd]{amsart}
    \fi

    \def\atSign{@@}

    \def\mathbb{\Bbb}
    \def\mathfrak{\frak}
    \def\mathbf{\bold}
    \ifx\boldsymbol\undefined   
      \def\boldsymbol#1{{\bold #1}}
    \fi

    \def\mathbit{\boldsymbol}

    \makeatletter
    \newenvironment{proof}{%
         \@ifnextchar[{%
                       \expandafter\let\expandafter\end@proof
                         \csname endpf*\endcsname
                         \my@proof
                      }{\let\end@proof\endpf\pf}%
        }{\end@proof}
    \def\my@proof[#1]{\@nameuse{pf*}{#1}}
    \makeatother

    \def\xrightarrow[#1]#2{@>{#2}>{#1}>}
    \def\xleftarrow[#1]#2{@<{#2}<{#1}<}

    \def\providecommand#1{\def#1}
    \def\emph#1{{\em #1}}
    \def\textbf#1{{\bf #1}}

    \def\mathring{\overset{\,\,{}_\circ}}
  \else
      \ifx\usepackage\RequirePackage
      \else
        \documentclass[12pt]{amsart}
        \usepackage{amssymb,amscd}
      \fi
      \ifx \mathring\undefined          
        \DeclareMathAccent{\mathring}{\mathalpha}{operators}{"17}
      \fi

      \long\def\FAKEendPROOF{\endtrivlist}
      \ifx\endproof\FAKEendPROOF
          \def\endproof{\qed\endtrivlist}
      \fi

      \ifx\mathbit\undefined    
        \DeclareMathAlphabet{\mathbit}{OML}{cmm}{b}{it}
      \fi

      \def\atSign{@}

      \advance\oddsidemargin -0.1\textwidth
      \evensidemargin\oddsidemargin
      \def\Sb#1\endSb{_{\substack{#1}}}
      \def\Sp#1\endSp{^{\substack{#1}}}

  \fi
\makeatletter
\@ifundefined{DeclareFontShape}
     {\@ifundefined{selectfont}
             {
                \message{We need AmS-LaTeX, this version of LaTeX does not
                        support it}\@@end
             }{
                \def\mathcal{\cal}
                
                \def\pcyr{%
                        \def\default@family{UWCyr}%
                        \let\oldSl@\sl
                        \def\sl{\def\default@shape{it}\oldSl@}%
                        \cyracc
                        \language\Russian\family{UWCyr}\selectfont
                }
             }}{
                \DeclareFontEncoding{OT2}{}{\noaccents@}
                \DeclareFontSubstitution{OT2}{cmr}{m}{n}
                \DeclareFontFamily{OT2}{cmr}{\hyphenchar\font45 }
                \DeclareFontShape{OT2}{cmr}{m}{n}{%
                     <5><6><7><8><9><10>gen*wncyr %
                     <10.95><12><14.4><17.28><20.74><24.88> wncyr10 %
                }{}
                \DeclareFontShape{OT2}{cmr}{m}{it}{%
                     <5><6><7><8><9><10> gen * wncyi%
                     <10.95><12><14.4><17.28><20.74><24.88> wncyi10%
                }{}
                \DeclareFontShape{OT2}{cmr}{bx}{n}{%
                     <5><6><7><8><9><10> gen * wncyb%
                     <10.95><12><14.4><17.28><20.74><24.88> wncyb10%
                }{}
                \DeclareFontShape{OT2}{cmr}{m}{sl}{%
                     <-> ssub * cmr/m/it%
                }{}
                \DeclareFontShape{OT2}{cmr}{m}{sc}{%
                     <5><6><7><8><9><10>%
                     <10.95><12><14.4><17.28><20.74><24.88> wncysc10%
                }{}
                \DeclareFontFamily{OT2}{cmss}{\hyphenchar\font45 }
                \DeclareFontShape{OT2}{cmss}{m}{n}{%
                     <8><9><10> gen * wncyss%
                     <10.95><12><14.4><17.28><20.74><24.88> wncyss10%
                }{}
                
                \def\cyrencodingdefault{OT2}
                \def\pcyr{%
                        \cyracc
                        \let\encodingdefault\cyrencodingdefault
                        \language\Russian\fontencoding{OT2}\selectfont
                }
             }

\@ifundefined{theorembodyfont}
     {
        \def\theorembodyfont#1{\relax}
        \makeatletter
          \let\@@th@plain\th@plain
          \def\th@plain{ \@@th@plain \slshape }
        \makeatother
        \let\normalshape\relax
     }{}

\ifx\cprime\undefined
     \def\cprime{$'$}
\fi
\makeatletter
  \def\@sect@my#1#2#3#4#5#6[#7]#8{%
\ifnum #2>\c@secnumdepth
   \let\@svsec\@empty
 \else
   \refstepcounter{#1}%
\edef\@svsec{\ifnum#2<\@m
             \@ifundefined{#1name}{}{\csname #1name\endcsname\ }\fi
\noexpand\rom{\csname the#1\endcsname.}\enspace}\fi
 \@tempskipa #5\relax
 \ifdim \@tempskipa>\z@ 
   \begingroup #6\relax
   \@hangfrom{\hskip #3\relax\@svsec}{\interlinepenalty\@M #8\par}%
   \endgroup
   \if@article\else\csname #1mark\endcsname{%
        \ifnum \c@secnumdepth >#2\relax\csname the#1\endcsname. \fi#7}\fi
\ifnum#2>\@m \else
       \let\@tempf\\ \def\\{\protect\\}\addcontentsline{toc}{#1}%
{\ifnum #2>\c@secnumdepth \else
             \protect\numberline{%
               \ifnum#2<\@m
               \@ifundefined{#1name}{}{\csname #1name\endcsname\ }\fi
               \csname the#1\endcsname.}\fi
           #8}\let\\\@tempf
     \fi
 \else
  \def\@svsechd{#6\hskip #3\@svsec
    \@ifnotempty{#8}{\ignorespaces#8\unskip
       \ifnum\spacefactor<1001.\fi}%
        \ifnum#2>\@m \else
          \let\@tempf\\ \def\\{\protect\\}\addcontentsline{toc}{#1}%
            {\ifnum #2>\c@secnumdepth \else
              \protect\numberline{%
                \ifnum#2<\@m
                \@ifundefined{#1name}{}{\csname #1name\endcsname\ }\fi
                \csname the#1\endcsname.}\fi
             #8}\let\\\@tempf\fi}%
 \fi
\@xsect{#5}}
  \ifx\RequirePackage\undefined
  \let\@sect\@sect@my             
  \fi
  \ifx\RequirePackage\undefined 
  \def\th@remark@my{\theorempreskipamount6\p@\@plus6\p@
    \theorempostskipamount\theorempreskipamount
    \def\theorem@headerfont{\it}\normalshape}
    \let\th@remark\th@remark@my
  \else                         
    \let\o@@remark\th@remark

    \ifx\theorempostskipamount\undefined                
    \else                                               
      \def\th@remark{\o@@remark
        \ifdim\theorempostskipamount < 2pt\relax
          \theorempostskipamount\theorempreskipamount
             \multiply\theorempostskipamount\tw@
             \divide\theorempostskipamount\thr@@
        \fi
      }
    \fi
  \fi
\let\myLabel\@gobble
\def\labelsONmargin{\@mparswitchfalse\def\myLabel##1{\@bsphack\marginpar
                                  {\normalshape\tiny\rm Label ##1}\@esphack}}
\ifx \url\undefined
  \def\url#1{{\tt #1}}%
\fi
\def\cyracc{\def\u##1{
                \if \i##1\char"1A%
                \else \if I##1\char"12%
                \else \accent"24 ##1\fi\fi }%
\def\"##1{\if e##1{\char"1B}%
                \else \if E##1{\char"13}%
                \else \accent"7F ##1\fi\fi }%
\def\9##1{\if##1z\char"19
\else\if##1Z\char"11
\else\if##1E\char"03
\else\if##1e\char"0B
\else\if##1u\char"18
\else\if##1U\char"10
\else\if##1A\char"17
\else\if##1a\char"1F
\else\if##1p\char"7E
\else\if##1P\char"5E
\else\if##1Q\char"5F
\else\if##1q\char"7F
\else\if##1i\char"1A
\else\if##1I\char"12
\else\if##1N\char"7D
\fi
\fi
\fi
\fi
\fi
\fi
\fi
\fi
\fi
\fi
\fi
\fi
\fi
\fi
\fi
}%
\def\cydot{{\kern0pt}}}%
\def\cydot{$\cdot$}

\ifx\Russian\undefined
        \def\Russian{0\relax
    \message{Don't know the hyphenation rules for Russian^^J
                        Please do INITeX with `input  russhyph' in the
                        command line}%
                \gdef\Russian{0\relax}%
        }
\fi
\ifx\RequirePackage\undefined                   
  \def\@putname#1#2#3#4{\def\@@ref{#3}\let\old@bf\bf
        \def\bf##1{\old@bf\if?\noexpand##1?{#4}\else##1\fi}%
        #1{#2}%
        \let\bf\old@bf}
\else                                           
  \def\@putname#1#2#3#4{\def\@@ref{#3}\let\old@bf\bf    
        \let\old@reset@font\reset@font                  
        \def\bf##1{\old@bf\if?\noexpand##1?{#4}\else##1\fi}%
        \def\reset@font##1##2{\old@reset@font##1\if?\noexpand##2?{#4}\else##2\fi}#1{#2}%
        \let\bf\old@bf\let\reset@font\old@reset@font}
\fi
\let\my@ref=\ref
\def\ref#1{\@putname\my@ref{#1}{#1}{\tiny\rm\@@ref}}
\let\my@pageref=\pageref
\def\pageref#1{\@putname\my@pageref{#1}{#1}{\tiny\rm\@@ref}}
\let\my@cite=\cite
\def\cite#1{\@putname\my@cite{#1}{\@citeb}{\tiny\rm\@@ref}}
\makeatother
\fi
\relax
\theoremstyle{plain} 
\theorembodyfont{\sl}

\numberwithin{equation}{section}
\theorembodyfont{\rm}
\theoremstyle{definition}

\newtheorem{definition}{Definition}[section]

\theoremstyle{remark}

\newtheorem{remark}[definition]{Remark} 

\theoremstyle{plain} 
\theorembodyfont{\sl}

\newtheorem{theorem}[definition]{Theorem}
\newtheorem{lemma}[definition]{Lemma}
\newtheorem{corollary}[definition]{Corollary}

\usepackage{xypic}

\begin{document}
\bibliographystyle{amsplain}
\relax

\def\AW[#1]^#2_#3{\ar@{-^>}@<.5ex>[#1]^{#2} \ar@{_<-}@<-.5ex>[#1]_{#3}}
 \def\NAW[#1]{\ar@{-^>}@<.5ex>[#1] \ar@{_<-}@<-.5ex>[#1]}

\title{ Classification of finite-growth \\
general Kac-Moody superalgebras }

\author{Crystal~Hoyt and Vera~Serganova} 

\date{ October 10, 2006 }

\let\thefootnote\relax

\footnotetext{Department of Mathematics, University of California, Berkeley, CA 94720, USA}
\footnotetext{\emph{Email:} crystal\atSign{}math.berkeley.edu, serganov\atSign{}math.berkeley.edu}
 
\maketitle

\begin{abstract} A contragredient Lie superalgebra is a superalgebra
  defined by a Cartan matrix. A contragredient Lie superalgebra has
  finite-growth if the dimensions of the graded components (in the natural
  grading)  depend polynomially on the degree. In this paper we
  classify finite-growth contragredient Lie superalgebras.
  Previously,
  such a classification was known only for the symmetrizable case.
\end{abstract}

\section{Introduction}

 Affine superalgebras have many interesting applications in
combinatorics, number theory and physics (see [3,6,7,8]).
Finite-growth contragredient Lie superalgebras without zeros on
the diagonal were classified by V.G. Kac in [3]. Finite-growth
symmetrizable were classified by J.W. van de Leur in [10,11]. In
the present paper we give a complete classification of
finite-growth contragredient Lie superalgebras without imposing
any conditions on the Cartan matrices. This list includes examples
previously known see [4,12] and some new superalgebras, however
these new superalgebras are not simple.

This also completes the classification of finite-growth Kac-Moody
superalgebras, where we find that there are only two
non-symmetrizable families, $q(n)^{(2)}$ and $S(1,2,a)$.  In
contrast to the Kac-Moody algebra situation, $S(1,2,a)$ is not a
central extension of a loop algebra.  However, it appears in the
list of conformal superalgebras [4].

The main result is as follows.

\begin{theorem} Let $\mathfrak{ g} (A)$ be a contragredient Lie
  superalgebra of finite growth and suppose the matrix $A$ is indecomposable
with no zero rows.  Then either $A$ is symmetrizable and
$\mathfrak{g}(A)$ is isomorphic to an affine or finite
dimensional Lie superalgebra classified in [10,11], or it is
listed in Table 8.1.

\end{theorem}

We also handle the case where the matrix $A$ has a zero row. In
this case, the algebra $\mathfrak{g}(A)$ is not simple and
basically is obtained by extending a finite dimensional algebra by
a Heisenberg algebra.

The main idea of the classification is to use odd reflections (see
[9]). Odd reflections are used to generalize the Weyl group for
superalgebras, allowing one to move from one base to another.  A Lie
superalgebra usually has more than one Cartan matrix, unlike the Lie
algebra case where the Cartan matrix is unique.

Our proof goes in the following way. We show that in a
finite-growth superalgebra, simple roots act locally nilpotently.
This implies certain conditions on a Cartan matrix (see Lemma
3.1). The crucial point of our classification is that these
conditions should still hold after we apply an odd reflection.
This gives rather strong conditions on a matrix, which allows us
to list them all. These conditions are only slightly weaker than
the Kac-Moody conditions.

The first author classifies in [1] all matrices with a zero on the
main diagonal satisfying the condition that the matrix and all its
odd reflections are generalized Cartan matrices (see conditions in
definition ~\ref{def1}). Remarkably, the superalgebras defined by
such matrices almost always have finite-growth.  The exception is a
certain family of matrices of size 3, which is described in [1]. By
comparing the classification in [1] to the finite-growth
symmetrizable classification in [10,11] we obtain the following

\begin{theorem} Let $A$ be a symmetrizable generalized Cartan matrix with
a zero on the main diagonal. Suppose that any matrix $A'$ obtained
  from $A$ by a sequence of odd reflections is again a generalized Cartan matrix. Then,
  $\mathfrak {g} (A)$ has finite growth.
\end{theorem}

Most of the calculations for our classification are in [1].

\section{Definitions }

Let $ I=\left\{1,\dots ,n\right\} $, $ p:I \to {\mathbb Z}_{2} $
and $ A=\left(a_{ij}\right) $ be a matrix. Fix a vector space $
{\mathfrak h} $ of dimension $ n+\operatorname{cork}\left(A\right)
$, linearly independent $ \alpha_{1},\dots
,\alpha_{n}\in{\mathfrak h}^{*} $ and $ h_{1},\dots
,h_{n}\in{\mathfrak h} $ such that $
\alpha_{j}\left(h_{i}\right)=a_{ij} $, define a Lie superalgebra $
\bar{{\mathfrak g}}\left(A\right) $ by generators $ X_{1},\dots
,X_{n},Y_{1},\dots ,Y_{n} $ and $ {\mathfrak h} $ with relations
\begin{equation}
\left[h,X_{i}\right]=\alpha_{i}\left(h\right)X_{i}\text{, }\left[h,Y_{i}\right]=-\alpha_{i}\left(h\right)Y_{i}\text{, }\left[X_{i},Y_{j}\right]=\delta_{ij}h_{i}.
\label{equ1}\end{equation}\myLabel{equ1,}\relax
Here we assume that all elements of $ {\mathfrak h} $ are even and $ p\left(X_{i}\right)=p\left(Y_{i}\right)=p\left(i\right) $.

By $ {\mathfrak g}\left(A\right) $ (or $ {\mathfrak g} $ when the Cartan matrix is fixed) denote the quotient
of $ \bar{{\mathfrak g}}\left(A\right) $ by the unique maximal ideal which intersects $ {\mathfrak h} $ trivially. It is
clear that if a matrix $ B=DA $ for some invertible diagonal $ D $ then
$ {\mathfrak g}\left(B\right)\cong{\mathfrak g}\left(A\right) $. Indeed an isomorphism can be obtained by mapping $ h_{i} $ to $ d_{i i}h_{i} $.
Therefore without loss of generality we may assume that $ a_{i i}=2 $ or 0.

The Lie superalgebra $ {\mathfrak g}={\mathfrak g}\left(A\right) $ has a {\em root\/} decomposition
\begin{equation}
{\mathfrak g}={\mathfrak h}\oplus\oplus_{\alpha\in\Delta}{\mathfrak g}_{\alpha}.
\notag\end{equation}
Clearly, one can define $ p:\Delta \to {\mathbb Z}_{2} $ by putting $ p\left(\alpha\right)=0 $ or 1 whenever $ {\mathfrak g}_{\alpha} $ is
even or odd, respectively. By $ \Delta_{0}\left(\Delta_{1}\right) $ we denote the set of even (odd) roots.
Every root
is a positive or negative linear combination of $ \alpha_{1},\dots ,\alpha_{n} $. According to
this we call a root positive or negative and have the decomposition
$ \Delta=\Delta^{+}\cup\Delta^{-} $. The roots $ \alpha_{1},\dots ,\alpha_{n} $ are called {\em simple roots}. Sometimes
instead of $ a_{ij} $ we will write $ a_{\alpha\beta} $ where $ \alpha=\alpha_{i} $, $ \beta=\alpha_{j} $.

One sees easily that there are the following possibilities for each
simple root $ \alpha=\alpha_{i}: $
\begin{enumerate}
\item if $ a_{\alpha\alpha}=2 $ and $ p\left(\alpha\right)=0 $,
then $ X_{\alpha} $, $ Y_{\alpha} $ and $ h_{\alpha} $ generate
the subalgebra isomorphic to $ sl\left(2\right) $;

\item if $ a_{\alpha\alpha}=0 $ and $ p\left(\alpha\right)=0 $,
then $ X_{\alpha} $, $ Y_{\alpha} $ and $ h_{\alpha} $ generate
the subalgebra isomorphic to the Heisenberg algebra;

\item if $ a_{\alpha\alpha}=2 $ and $ p\left(\alpha\right)=1 $,
then $
\left[X_{\alpha},X_{\alpha}\right]=\left[Y_{\alpha},Y_{\alpha}\right]\not=0
$ and $ X_{\alpha} $, $ Y_{\alpha} $ and $ h_{\alpha} $ generate
the subalgebra isomorphic to $ \operatorname{osp}\left(1|2\right)
$, in this case $ 2\alpha\in\Delta $;

\item if $ a_{\alpha\alpha}=0 $ and $ p\left(\alpha\right)=1 $,
then $
\left[X_{\alpha},X_{\alpha}\right]=\left[Y_{\alpha},Y_{\alpha}\right]=0
$ and $ X_{\alpha} $, $ Y_{\alpha} $ and $ h_{\alpha} $ generate
the subalgebra isomorphic to $ sl\left(1|1\right) $.
\end{enumerate}

In the last case we say that $ \alpha $ is {\em isotropic\/} and
by definition any isotropic root is odd. In the other cases a root
is called {\em non-isotropic}. A simple root $ \alpha $ is {\em
regular\/} if for any other simple root $ \beta $, $
a_{\alpha\beta}=0 $ implies $ a_{\beta\alpha}=0 $. Otherwise a
simple root is called {\em singular}.

\begin{lemma} \label{lm20}\myLabel{lm20}\relax  For any subset $ J\subset I $ the subalgebra $ {\mathfrak l} $ in $ {\mathfrak g}\left(A\right) $ generated by
$ {\mathfrak h} $, $ X_{i} $ and $ Y_{i} $, where $ i\in J $, is
isomorphic to $ {\mathfrak h}'\oplus{\mathfrak
g}\left(A_{J}\right) $, where $ A_{J} $ is the submatrix of $ A $
with coefficients $ \left(a_{ij}\right)_{i,j\in J} $ and $
{\mathfrak h}' $ is a subspace of $ {\mathfrak h} $. More
precisely $ {\mathfrak h}' $ is the maximal subspace in $
\cap_{i\in J} \operatorname{Ker} \alpha_{i} $ which trivially
intersects the span of $ h_{i} $, $ i\in J $.

\end{lemma}

\begin{proof} Let $ \bar{{\mathfrak l}} $ be the subalgebra of $ \bar{{\mathfrak g}}\left(A\right) $ generated by $ {\mathfrak h} $, $ X_{i},Y_{i} $,
$ i\in J $. Then $ \bar{{\mathfrak l}} $ is isomorphic to $ {\mathfrak h}'\oplus\bar{{\mathfrak g}}\left(A_{J}\right) $, since by construction $ {\mathfrak h}' $ lies in
the center of $ \bar{{\mathfrak l}} $, and $ \left[\bar{{\mathfrak l}},\bar{{\mathfrak l}}\right]\cap{\mathfrak h}'=\left\{0\right\} $. Let ${\mathcal I}$ be the maximal ideal in
$ \bar{{\mathfrak g}}\left(A\right) $ intersecting $ {\mathfrak h} $ trivially and $ {\mathcal J} $ be the ideal in $ \bar{{\mathfrak l}} $ intersecting $ {\mathfrak h} $
trivially. We have to show that $ {\mathcal J}={\mathcal I}\cap\bar{{\mathfrak l}} $. Obviously, $ {\mathcal I}\cap\bar{{\mathfrak l}}\subset{\mathcal J} $. On the other
hand, write $ {\mathcal J}={\mathcal J}^{+}\oplus{\mathcal J}^{-} $, note that $ \left[Y_{j},{\mathcal J}^{+}\right]=\left[X_{j},{\mathcal J}^{-}\right]=0 $ for any $ j\in I-J $.
Therefore, the ideal generated by $ {\mathcal J} $ in $ \bar{{\mathfrak g}}\left(A\right) $ intersects $ {\mathfrak h} $ trivially. That
implies $ {\mathcal J}\subset{\mathcal I}\cap\bar{{\mathfrak l}} $, and lemma is proven.\end{proof}

A superalgebra $ {\mathfrak g}={\mathfrak g}\left(A\right) $ has a
natural $ {\mathbb Z} $-grading $ {\mathfrak g}=\oplus{\mathfrak
g}_{n} $ if we put $ {\mathfrak g}_{0}:={\mathfrak h} $ and $
{\mathfrak g}_{1}={\mathfrak g}_{\alpha_{1}}\oplus\dots
\oplus{\mathfrak g}_{\alpha_{n}} $. This grading is called {\em
principal}. We say that $ {\mathfrak g} $ is of {\em finite
growth\/} if $ \dim  {\mathfrak g}_{n} $ grows polynomially
depending on $ n $. This means the Gelfand-Kirillov dimension of $
{\mathfrak g} $ is finite.  The following Lemma is a
straightforward corollary of Lemma~\ref{lm20}.

\begin{lemma} \label{lm21}\myLabel{lm21}\relax  If $ {\mathfrak g}\left(A\right) $ is of finite growth, then for any subset $ J \subseteq I $ the
Lie superalgebra $ {\mathfrak g}\left(A_{J}\right) $ is of finite
growth.

\end{lemma}

The matrix $ A $ is called {\em indecomposable\/} if the set $ I $
can not be decomposed into the disjoint union of non-empty subsets
$ J,K $ such that $ a_{ij}=a_{ji}=0 $ whenever $ i\in I,j\in K $.
We say that $ A $ is {\em elemental\/} if it has no zero rows.
Otherwise we call $ A $ {\em non-elemental}.
\section{Integrability and finite growth }

We say that $ {\mathfrak g}\left(A\right) $ is {\em integrable\/} if $ \operatorname{ad}_{X_{i}} $ are locally
nilpotent for all $ i\in I $. In this case $ \operatorname{ad}_{Y_{i}} $ are also locally nilpotent.

\begin{lemma} \label{lm23}\myLabel{lm23}\relax  Let $ {\mathfrak g}\left(A\right) $ be integrable and let
$ A $ be elemental, indecomposable with $ n\geq2 $. Then after
rescaling the rows $ A $
satisfies the following conditions
\begin{enumerate}
\item for any $ i\in I $ either $ a_{i i}=0 $ or $ a_{i i}=2 $;
\item if $ a_{i i}=0 $ then $ p\left(i\right)=1 $; \item if $ a_{i
i}=2 $ then $ a_{ij}\in2^{p\left(i\right)}{\mathbb Z}_{-} $;
 \item
if $ a_{ij}=0 $ and $ a_{ji}\not=0 $, then $ a_{i i}=0 $.
\end{enumerate}
\end{lemma}

\begin{proof} If $ a_{i i}=0 $ and $ p\left(i\right)=0 $, then $ h_{i},X_{i},Y_{i} $ generate the
Heisenberg subalgebra $ {\mathfrak k} $. There exists $ j $ such that $ a_{ij}\not=0 $. Then $ Y_{j} $
generates the irreducible infinite dimensional $ {\mathfrak k} $-submodule with central
charge $ -a_{ij} $. Hence $ \left(\operatorname{ad}_{Y_{i}}\right)^{m}Y_{j}\not=0 $ for all $ m>0 $. This proves the first
statement.

If $ a_{i i}=2 $, then $ h_{i},X_{i},Y_{i} $ generate the subalgebra $ {\mathfrak k} $ isomorphic to $ sl_{2} $
for even $ i $ and $ \operatorname{osp}\left(1|2\right) $ for odd $ i $. Any $ Y_{j} $ generates a $ {\mathfrak k} $-submodule $ M $ with
highest weight $ -a_{ij} $. If $ \operatorname{ad}_{Y_{i}} $ is locally nilpotent, this submodule must be
finite-dimensional. From elementary representation theory we know that
this implies $ -a_{ij}\in{\mathbb Z}_{\geq0} $ for $ sl_{2} $ and $ -a_{ij}\in2{\mathbb Z}_{\geq0} $ for $ \operatorname{osp}\left(1|2\right) $. The second
statement is proven.

To prove the last statement assume the opposite, i.e. $ a_{i i}=2 $, $ a_{ij}=0 $,
$ a_{ji}\not=0 $. Let $ {\mathfrak k} $ and $ M $ be as in the previous paragraph. Since $ -a_{ij}=0 $, $ M $
has the highest weight 0, therefore $ M $ is a trivial $ {\mathfrak k} $-module or a Verma
module. Since $ \operatorname{ad}_{Y_{i}} $ is locally nilpotent, then $ M $ is trivial. Hence
$ \left[Y_{i},Y_{j}\right]=0 $. But
\begin{equation}
\left[X_{j},\left[Y_{i},Y_{j}\right]\right]=\pm\left[H_{j},Y_{i}\right]=\pm a_{ji}Y_{i}\not=0.
\notag\end{equation}
Contradiction.\end{proof}

\begin{theorem} \label{th1}\myLabel{th1}\relax  Let A be elemental. If $ {\mathfrak g}\left(A\right) $ has finite growth,
then $ {\mathfrak g}\left(A\right) $ is integrable.

\end{theorem}

\begin{proof} We start with the following

\begin{lemma} \label{lm24}\myLabel{lm24}\relax  Let $ n=2 $ and $ a_{12} a_{21}\not=0 $. If $ \operatorname{ad}_{X_{1}} $ does not act
nilpotently on $ X_{2} $, then $ {\mathfrak g}\left(A\right) $ has infinite growth.

\end{lemma}

\begin{proof} We use the fact that $ A $ is symmetrizable, therefore $ {\mathfrak g}\left(A\right) $
admits an invariant symmetric form ( , ). By the condition of Lemma
$ \left(\operatorname{ad}_{X_{1}}\right)^{m}X_{2}\not=0 $ for any $ m $. Let
\begin{equation}
E_{k}=\left(\operatorname{ad}_{X_{1}}\right)^{3k}X_{2}\text{, }F_{k}=\left(\operatorname{ad}_{Y_{1}}\right)^{3k}Y_{2}\text{, }H_{k}=\left[E_{k},F_{k}\right]\text{, }\beta_{k}=3k\alpha_{1}+\alpha_{2}\text{, }k\in{\mathbb Z}_{\geq0}.
\notag\end{equation}
Obviously, $ \beta_{k}-\beta_{l}\notin\Delta $ if $ k\not=l $. Hence
\begin{equation}
\left[E_{k},F_{l}\right]=\delta_{kl}H_{k}. \notag\end{equation} Let
$ {\mathfrak l} $ be the Lie subalgebra generated by $
H_{k},E_{k},F_{k} $ with $ k\in{\mathbb Z}_{+} $. We claim that $
{\mathfrak l} $ has infinite growth with respect to the grading
induced by the principal grading of $ {\mathfrak g}\left(A\right) $,
and therefore $ {\mathfrak g}\left(A\right) $ has infinite growth.

To prove our claim we will show first that there exists $ k $ such that
\begin{equation}
\left(\beta_{k_{1}},\beta_{k_{2}}+\dots +\beta_{k_{r}}\right)\not=0
\label{equ21}\end{equation}\myLabel{equ21,}\relax
for any $ k\leq k_{1}<\dots <k_{r} $. The claim follows immediately from the following

\begin{lemma} \label{lm245}\myLabel{lm245}\relax  Let $ V $ be a two-dimensional space with non-zero
symmetric bilinear product $ \left(\cdot,\cdot\right) $. For any two
vectors $ v,w\in V $ with $ \left(v,w\right)\not=0 $, there exist $ k>0 $ and $ c\in{\mathbb C} $ such that
\begin{equation}
\operatorname{Re}\text{ }  c\left(v+pw,v+qw\right)>0
\notag\end{equation} for all $ p,q\geq k $.

\end{lemma}

\begin{proof} Assume first that $ \left(w,w\right)\not=0 $. Put $ c=\frac{1}{\left(w,w\right)} $. Then
\begin{equation}
\lim _{p,q\to\infty}\frac{c\left(v+pw,v+qw\right)}{pq}=1.
\notag\end{equation}

Now let $ \left(w,w\right)=0 $. Put $ c=\frac{1}{\left(v,w\right)} $. Then again
\begin{equation}
\lim _{p,q\to\infty}\frac{c\left(v+pw,v+qw\right)}{pq}=1,
\notag\end{equation} and Lemma~\ref{lm245} is proven.\end{proof}

Recall that $ \left(\beta_{k_{2}}+\dots +\beta_{k_{r}}\right)\left(H_{k_{1}}\right) $ is proportional to
$ \left(\beta_{k_{1}},\beta_{k_{2}}+\dots +\beta_{k_{r}}\right) $ with non-zero coefficient and therefore it is not zero.
Hence
\begin{equation}
\left[F_{k_{1}},\left[E_{k_{1}}\left[E_{k_{2}}\dots E_{k_{r}}\right]\right]\right]=\left[H_{k_{1}},\left[E_{k_{2}}\dots E_{k_{r}}\right]\right]=\left(\beta_{k_{2}}+\dots +\beta_{k_{r}}\right)\left(H_{k_{1}}\right)\left[E_{k_{2}}\dots E_{k_{r}}\right]\not=\text{
0}
\notag\end{equation}
if $ k\leq k_{1}<\dots <k_{r} $.

The last calculation shows that
the commutator $ \left[E_{k_{1}},\dots ,E_{k_{r}}\right] $ is not zero. Moreover, the same
calculation shows that all such commutators are linearly independent.

\begin{lemma} \label{lm246}\myLabel{lm246}\relax  Let $ m $ be a positive integer, and let $ f_{m}\left(n\right) $ be the number of
ways to write $ n $ as the sum $ m_{1}+\dots +m_{r} $ with $ m_{r}>m_{r-1}>\dots >m_{1}\geq m $.
Then $ f_{m}\left(n\right) $ grows exponentially on $ n $.

\end{lemma}

Finally, introduce a new $ {\mathbb Z} $-grading on $ {\mathfrak
g}\left(A\right) $, by putting $ \deg ' X_{1}=1 $, $ \deg '
X_{2}=3 $. One can see that
\begin{equation}
\deg ' \left[E_{k_{1}},\dots
,E_{k_{r}}\right]=3\left(k_{1}+1\right)+\dots
+3\left(k_{r}+1\right). \notag\end{equation} Therefore $ \dim
{\mathfrak g}\left(A\right)'_{3n}\geq f_{k+1}\left(n\right) $. Hence
$ {\mathfrak g}\left(A\right) $ has infinite growth in this new
grading. However $ |\deg '\left(X\right)|\leq3 |\deg \left(X\right)|
$ for any $ X\in{\mathfrak g}\left(A\right) $, therefore $
{\mathfrak g}\left(A\right) $ has infinite growth in the principal
grading. Hence, Lemma~\ref{lm24} is proven.\end{proof}

\begin{lemma} \label{lm25}\myLabel{lm25}\relax  Let $ n=2 $, $ a_{11}=2 $, $ a_{12}=0 $, $ a_{21}\not=0 $. Then $ {\mathfrak g}\left(A\right) $ is of
infinite growth.

\end{lemma}

\begin{proof} Let $ X_{12}=\left[X_{1},X_{2}\right] $, $ Y_{12}=\frac{\left(-1\right)^{p\left(1\right)p\left(2\right)}}{a_{21}}\left[Y_{1},Y_{2}\right] $. Then one can
check easily the following relations
\begin{equation}
\left[X_{12},Y_{12}\right]=\left[X_{1},Y_{1}\right]=H_{1}\text{,
}\left[X_{12},Y_{1}\right]=\left[Y_{12},X_{1}\right]=0.
\notag\end{equation} Then $ X_{12},Y_{12},X_{1} $, $ Y_{1} $ and $
{\mathfrak h} $ generate a subalgebra isomorphic to a quotient of
$ \bar{{\mathfrak g}}\left(B\right) $ for $
b_{11}=b_{12}=b_{21}=b_{22}=2 $. In particular, $ad_{X_1}$ is not
nilpotent.  By Lemma~\ref{lm24} such an algebra
has infinite growth, therefore $ {\mathfrak g}\left(A\right) $ has
infinite growth.\end{proof}

\begin{lemma} \label{lm26}\myLabel{lm26}\relax  Let $ n=3 $, $ a_{21}=a_{22}=0 $ and $ a_{23}\not=0 $.
If $ \operatorname{ad}_{X_{1}} $ does not act
nilpotently on $ X_{2} $, then $ {\mathfrak g}\left(A\right) $ has infinite growth.

\end{lemma}

\begin{proof} Note that $ X_{1},X_{2} $, $ Y_{1},Y_{2} $ generate the subalgebra isomorphic
to $ {\mathfrak g}\left(A_{\left\{1,2\right\}}\right) $. However $ A_{\left\{1,2\right\}} $ is non-elemental, and by Lemma~\ref{lm22},
$ {\mathfrak g}\left(A_{\left\{1,2\right\}}\right) $ contains an infinite Heisenberg Lie superalgebra $ {\mathfrak l}={\mathfrak l}_{-1}\oplus{\mathbb C}h_{2}\oplus{\mathfrak l}_{1} $.
Note that $ \operatorname{ad}_{{\mathfrak l}_{1}}Y_{3}=0 $ and $ \operatorname{ad}_{h_{2}}\left(Y_{3}\right)=a_{23}Y_{3} $. Therefore the $ {\mathfrak l} $-submodule $ M $
generated by $ Y_{3} $ is isomorphic to $ U\left({\mathfrak l}\right)\otimes_{U\left({\mathbb C}h_{2}\oplus{\mathfrak l}_{1}\right)}{\mathbb C}Y_{3} $. It is not difficult
to see that $ M $ has infinite growth in the principal grading. Hence $ {\mathfrak g}\left(A\right) $
has infinite growth.\end{proof}

Now we can prove the theorem. Assume that $ \operatorname{ad}_{X_{i}} $ is not locally
nilpotent.
Then $ \operatorname{ad}_{X_{i}} $ does not act nilpotently on some $ X_{j} $. If $ a_{ij}=a_{ji}=0 $ then
$ \left[X_{i},X_{j}\right]=0 $. Therefore either $ a_{ij}\not=0 $ or $ a_{ji}\not=0 $. Consider the following
cases
\begin{enumerate}
\item
If $ a_{ij}a_{ji}\not=0 $, then $ A_{\left\{i,j\right\}} $ satisfies the conditions of Lemma~\ref{lm24},
and therefore $ {\mathfrak g}\left(A_{\left\{i,j\right\}}\right) $ has infinite growth;
\item
If $ a_{i i}=2,a_{ij}=0 $, then $ A_{\left\{i,j\right\}} $ satisfies the conditions of Lemma~%
\ref{lm25}, and therefore $ {\mathfrak g}\left(A_{\left\{i,j\right\}}\right) $ has infinite growth;
\item
The case $ a_{i i}=a_{ij}=0 $ is impossible since Lemma~\ref{lm22} implies
$ \operatorname{ad}_{X_{i}}^{2}X_{j}=0 $;
\item
If $ a_{i i}=2,a_{ij}\not=0 $, $ a_{ji}=a_{j j}=0 $ then $ a_{jk}\not=0 $ for some $ k $ since
$ A $ is non-elemental. Then by Lemma~\ref{lm26} $ {\mathfrak g}\left(A_{\left\{i,j,k\right\}}\right) $
has infinite growth;
\item
If $ a_{i i}=a_{jj}=2,a_{ij}\not=0 $, $ a_{ji}=0 $, then $ A_{\left\{i,j\right\}} $ satisfies the conditions of
Lemma~\ref{lm25}, and therefore $ {\mathfrak g}\left(A_{\left\{i,j\right\}}\right) $ has infinite growth;
\item
If $ a_{i i}=0,a_{ij}\not=0 $, $ a_{ji}=a_{jj}=0 $ then by the same argument as in the
fourth case $ {\mathfrak g}\left(A_{\left\{i,j,k\right\}}\right) $ has infinite growth for some $ k $;
\item
If $ a_{i i}=a_{ji}=0,a_{ij}\not=0 $, $ a_{jj}=2 $ then $ A_{\left\{i,j\right\}} $ satisfies the conditions of
Lemma~\ref{lm25}, and therefore $ {\mathfrak g}\left(A_{\left\{i,j\right\}}\right) $ has infinite growth.
\end{enumerate}

Thus, we have shown that $ {\mathfrak g}\left(A\right) $ contains a subalgebra of infinite
growth. Therefore $ {\mathfrak g}\left(A\right) $ itself has infinite growth. Contradiction.\end{proof}

\section{Odd reflections}
Suppose $
a_{kk}=0 $ and $ p(\alpha_{k})=1 $. Define $ r_{k}(\alpha_{i}) $
by the following formula
\begin{equation} r_{k}(\alpha_{k})=-\alpha_{k}, \notag \end{equation}
\begin{equation} r_{k}(\alpha_{i})=\alpha_{i} \text{ if } a_{ik}=a_{ki}=0, i \neq k, \notag \end{equation}
\begin{equation} r_{k}(\alpha_{i})=\alpha_{i}+\alpha_{k} \text{ if } a_{ik} \neq 0 \text{ or } a_{ki} \neq 0,
 i \neq k. \notag \end{equation}
Define
\begin{align*}
 X_{i}' &= \begin{cases}
    [X_{i},X_{k}] &\text{ if }i \neq k\text{ and }a_{ik} \neq 0\text{ or }a_{ki} \neq 0 \\
    X_{i} &\text{ if }i \neq k\text{ and }a_{ik}=a_{ki}=0 \\
    Y_{i} &\text{ if }i=k\end{cases}
    \\
Y_{i}' &= \begin{cases}
    [Y_{i},Y_{k}] &\text{ if }i \neq k\text{ and }a_{ik} \neq 0\text{ or }a_{ki} \neq 0 \\
    Y_{i} &\text{ if }i \neq k\text{ and }a_{ik}=a_{ki}=0 \\
    X_{i} &\text{ if } i=k  \end{cases}
    \\
h_{i}' &= [X_{i}',Y_{i}'].
\end{align*}
Then we have $ h_{i}' = \begin{cases}
    (-1)^{p(\alpha_{i})} (a_{ik}h_{k}+a_{ki}h_{i}) &\text{ if }i \neq k\text{ and }a_{ik} \neq 0\text{ or }a_{ki} \neq 0 \\
    h_{i} &\text{ if }i \neq k\text{ and }a_{ik}=a_{ki}=0 \\
    h_{k} &\text{ if } i=k  \end{cases} $

\begin{lemma} \label{lm1}\myLabel{lm1}\relax  The roots $ r_{i}\left(\alpha_{1}\right),\dots ,r_{i}\left(\alpha_{n}\right) $ are linearly independent. The
elements $ X'_{1},\dots ,X'_{n},Y'_{1},\dots ,Y'_{n} $ together with $ h'_{1},\dots ,h'_{n} $ satisfy~\eqref{equ1}.
Moreover if $ \alpha_{i} $ is regular, then $ {\mathfrak h} $ and $ X'_{1},\dots ,X'_{n},Y'_{1},\dots ,Y'_{n} $ generate $ {\mathfrak g} $.
\end{lemma}

\begin{proof} We have to check that $ \left[X'_{j},Y'_{k}\right]=0 $ for any $ j\not=k $. If $ j,k\not=i $,
then $ r_{i}\left(\alpha_{j}\right)-r_{i}\left(\alpha_{k}\right)\notin\Delta $, that forces $ \left[X'_{j},Y'_{k}\right]=0 $. So assume that $ j=i $. Then
for some constant $ c $ we have
\begin{equation}
\left[X'_{i},Y'_{k}\right]=c\left[Y_{i},\left[Y_{i},Y_{k}\right]\right]=0,
\notag\end{equation}
as $ \operatorname{ad}_{Y_{i}}^{2}=0 $. Similarly one deals with the case $ k=i $.

To prove the second statement note that if $ r_{i}\left(\alpha_{j}\right)=\alpha_{j}+\alpha_{i} $, then
\begin{equation}
X_{j}=\frac{1}{a_{ij}}\left[Y_{i},\left[X_{i},X_{j}\right]\right]=d\left[X'_{i},X'_{j}\right]
\notag\end{equation}
for some non-zero constant $ d $. Similarly, $ Y_{j}=d'\left[Y'_{i},Y'_{j}\right] $. If $ r_{i}\left(\alpha_{j}\right)=\alpha_{j} $, then
$ X_{j}=X'_{j} $, $ Y_{j}=Y'_{j} $. Finally, $ X_{i}=Y'_{i} $ and $ Y_{i}=X'_{i} $. Hence every generator $ X_{j},Y_{j} $
can be expressed in terms of $ X'_{1},\dots ,X'_{n},Y'_{1},\dots ,Y'_{n} $.\end{proof}

We see from Lemma~\ref{lm1} that if one has a matrix $ A $ and
some regular isotropic simple root $ \alpha_{i} $ for
$\mathfrak{g}(A)$, then one can construct another matrix $ A' $
such that $ {\mathfrak g}\left(A'\right)\cong{\mathfrak
g}\left(A\right) $. In this case we say that $ A' $ is obtained
from $ A $ and a base $ \Pi'=\left\{\alpha'_{1},\dots
,\alpha'_{n}\right\} $ is obtained from $
\Pi=\left\{\alpha_{1},\dots ,\alpha_{n}\right\} $ by the {\em
regular odd reflection\/} with respect to $ \alpha_{i} $. If $
\left(\Delta^{+}\right)' $ denotes the set of positive roots with
respect to $ \Pi' $, then one can check easily that
\begin{equation}
\left(\Delta^{+}\right)'=\Delta^{+}\cup\left\{-\alpha_{i}\right\}-\left\{\alpha_{i}\right\}.
\label{equ2}\end{equation}\myLabel{equ2,}\relax
Thus, any two sets of positive roots differ only by finite set of
isotropic roots. In particular, the set of even positive roots $ \Delta_{0}^{+} $ is
defined independently of a choice of a base $ \Pi $.

After rescaling the rows of the matrix $ A' $ we have that it is
defined by the following, (where we assume $ i \neq k $ and $ j
\neq k $),
\begin{align*}
    a_{kk}' &= a_{kk} \\
    a_{kj}' &= a_{kj} \\
    a_{ik}' &= -a_{ki}a_{ik} \\
    a_{ij}' &= \begin{cases}
        a_{ij} &\text{ if } a_{ik}=a_{ki}=0 \\
        a_{ki}a_{ij} &\text{ if }a_{ik}\neq 0 \text{ or } a_{ki} \neq 0, \text{ and } a_{kj}=a_{jk}=0 \\
        a_{ki}a_{ij}+a_{kj}a_{ik}+a_{ki}a_{ik} &\text{ if }a_{ik}\neq 0 \text{ or } a_{ki} \neq 0, \text{ and } a_{jk}\neq 0 \text{ or } a_{kj} \neq 0.
        \end{cases}
\end{align*}

\begin{remark} In practice, we rescale the rows of the matrix $A'$ so that the nonzero diagonal entries are equal to 2.
\end{remark}

\begin{remark} \label{rem1}\myLabel{rem1}\relax  If $ \alpha_{i} $ is a singular isotropic root,
then the corresponding odd reflection is called {\em singular\/}.
In this case, we can only show that $ {\mathfrak g}\left(A'\right)
$ is a subalgebra in $ {\mathfrak g}\left(A\right) $. Though, the
roots $ r_{i}\left(\alpha_{1}\right),\dots
,r_{i}\left(\alpha_{n}\right) $ still form a base in the geometric
sense, i.e. every $ \alpha\in\Delta $ is a linear combination $
\sum_{j=1}^{n} m_{j}r_{i}\left(\alpha_{j}\right) $, where all $
m_{j} $ are non-negative integers or all $ m_{j} $ are
non-positive integers. The set $ \Pi' $ obtained from $ \Pi $ by a
sequence of odd reflections, singular or regular, is called a {\em
geometric base.\/}
\end{remark}

\begin{lemma}
Suppose $\beta$ is a regular isotropic root, and let $A''$ and $X''_{\alpha}$, $Y''_{\alpha}$, $H''_{\alpha}$ denote the matrix and
generators obtained by reflecting twice at $\beta$.  Then by rescaling the rows of $A''$ and the generators $X''_{\alpha}$
accordingly we obtain the original matrix and generators.  Thus, $r_{\beta}\circ r_{\beta}\sim id$.
\end{lemma}
\begin{proof}
First observe that $a'_{ki}=a_{ki}$ and $a_{ki}=0 \Leftrightarrow a_{ik}=0$.
Then there are three cases to check. First, $r_{\beta}(r_{\beta}(\beta))=r_{\beta}(-\beta)=\beta$ and
$r_{\beta}(r_{\beta}(X_{\beta}))=r_{\beta}(Y_{\beta})=X_{\beta}$.  Second, if $a_{ki}=0$ then
$r_{\beta}(r_{\beta}(\alpha))=r_{\beta}(\alpha)=\alpha$ and
$r_{\beta}(r_{\beta}(X_{\alpha}))=r_{\beta}(X_{\alpha})=X_{\alpha}$.  Finally, if $a_{ki} \neq 0$ then $a'_{ki} \neq 0$,
and we have $r_{\beta}(r_{\beta}(\alpha))=r_{\beta}(\alpha+\beta)=(\alpha+\beta)-\alpha=\alpha$ and
$r_{\beta}(r_{\beta}(X_{\alpha})=r_{\beta}([X_{\alpha},X_{\beta}])=[[X_{\alpha},X_{\beta}],Y_{\beta}]=-a_{ki}X_{\alpha}$.
\end{proof}

\begin{lemma} \label{lm240}\myLabel{lm240}\relax  Let $ A' $ be obtained from $ A $ by an odd reflection. If $ A $
is elemental (non-elemental), then $ A' $ is elemental (non-elemental).
\end{lemma}

\begin{proof} If the $ i $-th row of $ A $ is zero, then an odd reflection
does not change $ \alpha_{i} $ and $ h_{i} $ and the $ i $-th row of $ A' $ is also zero. Suppose
the $ i $-th row of $ A' $ is zero. Then $ \alpha'_{j}\left(h'_{i}\right)=0 $ for all $ j $. Let $ A' $ be
obtained from $ A $ by an odd reflection $ r_{j} $. Then if $ a_{ji}\not=0 $
\begin{equation}
h_{i}^{'}=c\left(a_{ij}h_{j}+a_{ji}h_{i}\right)=c a_{ji}h_{i}
\notag\end{equation}
for some non-zero $ c $. If $ a_{ji}=0 $, then $ h'_{i}=h_{i} $. Any way, the $ i $-th row of $ A $
is zero.\end{proof}

Note that the notion of finite
growth does not depend on choice of a base as follows from

\begin{lemma} \label{lm2}\myLabel{lm2}\relax  Let $ \Pi' $ be obtained from $ \Pi $ by an odd reflection. Then $ \dim
{\mathfrak g}'_{n}\leq\dim  {\mathfrak g}_{-2n}\oplus\dots \oplus{\mathfrak g}_{2n} $. In particular, if $ {\mathfrak g}\left(A\right) $ is of finite growth, then
$ {\mathfrak g}\left(A'\right) $ is of finite growth.
\end{lemma}

\begin{theorem} \label{cor32}\myLabel{cor32}\relax  If $ {\mathfrak g}\left(A\right) $ has finite growth and A is
elemental, then matrix $A$ and any matrix $ A' $ obtained from A
by a sequence of odd reflections (singular or regular) satisfies
the matrix conditions of Lemma~\ref{lm23}.
\end{theorem}
\begin{proof}
This follows immediately from Lemma~\ref{lm2}, Lemma~\ref{lm23} and Theorem~\ref{th1}.
\end{proof}
 \begin{definition}\label{def1}\myLabel{def1}\relax
The matrix $ A $ is called a {\em generalized Cartan matrix\/} if
the following conditions hold
\begin{enumerate}
\item for any $ i\in I $ either $ a_{i i}=0 $ or $ a_{i i}=2 $;
\item if $ a_{i i}=0 $ then $ p\left(i\right)=1 $; \item if $ a_{i
i}=2 $ then $ a_{ij}\in2^{p\left(i\right)}{\mathbb Z}_{-} $; \item
if $ a_{ij}\not=0 $ then $ a_{ji}\not=0 $.
\end{enumerate}
\end{definition}

Note that the above conditions are just a little bit stronger than
conditions of Lemma~\ref{lm23}. More precisely

\begin{lemma} \label{lm27}\myLabel{lm27}\relax  A matrix satisfying the conditions of Lemma
~\ref{lm23} is a generalized Cartan matrix if and only if $\Pi$ does
not contain a singular root.

\end{lemma}

We call $ {\mathfrak g}\left(A\right) $ a {\em regular Kac-Moody
superalgebra\/} if $ A $ itself and any matrix obtained from $ A $
by a sequence of odd reflections is a generalized Cartan matrix.

\begin{corollary}
If $A$ is symmetrizable with a zero on the diagonal and $\mathfrak
{g}(A)$ has finite growth, then ${\mathfrak g} (A)$ is regular Kac-Moody.
\end{corollary}

\begin{proof}
This is an immediate consequence of Theorem~\ref{cor32} and Lemma~\ref{lm27}.
\end{proof}

\section{principal roots}

We call $ \alpha\in\Delta $ a {\em principal\/} root if $ \alpha $
belongs to some base $ \Pi' $ obtained from an initial base $ \Pi
$ by a sequence of regular odd reflections. A principal root
defines a subalgebra of $ {\mathfrak g} $ generated by $
X_{\alpha},Y_{\alpha} $ and $
h_{\alpha}=\left[X_{\alpha},Y_{\alpha}\right] $. We choose $
X_{\alpha},Y_{\alpha} $ and $ h_{\alpha} $ in such a way that $
\alpha\left(h_{\alpha}\right)=0 $ or 2. So the choice of these
generators is unique up to proportionality. For two roots $ \alpha
$ and $ \beta $ put $ a_{\alpha\beta}=\beta\left(h_{\alpha}\right)
$.

We define the subset $ \Pi_{0} \subseteq \Delta $ as follows. Let
$ \alpha\in\Pi_{0} $ if $ \alpha $ is even and belongs to some
geometric base, or $ \alpha=2\beta $, where $ \beta $ is odd,
non-isotropic and belongs to some geometric base. Obviously, $
\Pi_{0}\subset\Delta_{0}^{+} $. For a Lie superalgebra $l$ we
denote by $l'$ its commutator $[l,l]$

\begin{lemma} \label{lm4}\myLabel{lm4}\relax Let $S$ be a
  subset of $\Pi_0$, and let ${\mathfrak g}_S$ be the subalgebra of ${\mathfrak g}$ generated by $
  X_{\beta},Y_{\beta} $ for all $\beta \in S$ and $ h_{\beta}
  =[X_{\beta}, Y_{\beta}]$.
These elements satisfy the following
relations for a certain matrix $B$:
$$[h_{\beta}, X_{\gamma}]=b_{\beta , \gamma} X_{\beta},
[h_{\beta}, Y_{\gamma}]=-b_{\beta , \gamma} Y_{\beta}, [X_{\beta},
Y_{\gamma}]={\delta}_{\beta , \gamma} h_{\beta},\text{ for } \beta,\gamma\in{S}.
$$
There exists a surjective homomorphism ${\mathfrak g}_S \to
{\mathfrak
  g}'(B) /c$ where $c$ is some central subalgebra of the Cartan
subalgebra in ${\mathfrak g}' (B)=[{\mathfrak g}(B), {\mathfrak g} (B)]$.
\end{lemma}

\begin{proof} In order to check the relations it suffices to show that
$ \beta-\gamma\notin\Delta $ for any $ \beta,\gamma\in\Pi_{0} $.
Indeed, assume that $ \alpha=\beta-\gamma $. Then $
\alpha\in\Delta_{0} $ and without loss of generality we may assume
that $ \alpha $ is positive. Thus, we have $ \beta=\alpha+\gamma$.
On the other hand, there exists a geometric base $ \Pi' $ such
that either $ \beta\in\Pi' $ or $ \frac{\beta}{2}\in\Pi' $.
Anyway, $ \beta $ can not be decomposed into a sum of some
positive roots not proportional to itself. Contradiction.

Now let us prove the second statement. There exists a surjective
homomorphism $j: \bar{\mathfrak g}'(B) \to {\mathfrak g}_S$. Let
$i: \bar{\mathfrak g}'(B) \to {\mathfrak g}'(B)$ be the standard
projection. Let $\mathcal{I},\mathcal{J}$ be the kernels of $i,j$ respectively. Then
(see [5]) $\mathcal{I}={\mathcal{I}}^+ \oplus {\mathcal{I}}^-$, where ${\mathcal{I}}^{\pm}$ is the
unique maximal ideal in $\bar {\mathfrak n} ^{\pm}$. We claim that
$\mathcal{J}={\mathcal{J}}^+ \oplus {\mathcal{J}}^0 \oplus {\mathcal{J}}^-$, where ${\mathcal{J}}^{\pm}$ is an ideal
in $\bar {\mathfrak n} ^{\pm}$ and ${\mathcal{J}}^0$ is some ideal in the
Cartan subalgebra of $\bar{\mathfrak
  g}'(B)$. Indeed, let us choose $h\in \mathfrak h$ such that $\beta
(h)>0$ for all $\beta \in S$ and $\bar {h}$ in the Cartan
subalgebra of $\bar{\mathfrak g}(B)$ such that $\beta (\bar
{h})=\beta (h)$ for all $\beta \in S$.  Then one can
extend $j$ to the surjection
$$
\mathbb C \bar{h} +\bar{\mathfrak g}'(B) \to \mathbb C h
+{\mathfrak g}_S
$$
by putting $j(\bar h)=h$. If $h\in{\mathfrak g}_S$ we take $\bar h \in j^{-1}(h)$,
and then we do not need to extend $j$ since it is already defined on $\bar{h}$.
Since $\mathcal {J}$ is also an ideal in $ \mathbb C \bar{h} +\bar{\mathfrak
g}'(B)$, we have the decomposition $\mathcal{J}=\mathcal{J}^+ \oplus \mathcal{J}^0 \oplus \mathcal{J}^-$,
where $\mathcal{J}^{\pm}$ is the sum of the eigenspaces of ad $(\bar h)$
with  positive (negative) eigenvalues and $\mathcal{J}^0$ is the zero
eigenspace, hence $\mathcal{J}^0$ is an ideal in the Cartan subalgebra. Now
the maximality of $\mathcal{I}^{\pm}$ implies $\mathcal{J}^{\pm} \subset \mathcal{I}^{\pm}$. Let
$c=i(\mathcal{J}^0)$. obviously, $c$ commutes with all $X_{\beta},
Y_{\beta}$, hence $c$ is in the center of ${\mathfrak g}'(B)$.
Then clearly the surjective homomorphism ${\mathfrak g}_S \to
{\mathfrak
  g}'(B) /c$ exists.

\end{proof}
\begin{corollary} \label{cor1}\myLabel{cor1}\relax  If a Lie
  superalgebra $ {\mathfrak g}\left(A\right) $ has finite growth, then
  for any finite subset $S$ of ${\Pi}_0$
the Lie algebra $ {\mathfrak g}\left(B\right) $ also has finite
growth. In particular, $ B $ is an even generalized Cartan matrix.

\end{corollary}

\begin{remark}
We denote by $\mathfrak{g}(\overline{B})$ the algebra given by all roots in $\Pi_0$.
\end{remark}

Generally speaking it is possible that $ \Pi_{0} $ is infinite, but
this is certainly not possible if $ \mathfrak{g}(A) $ is regular and
has finite growth.

\begin{lemma} \label{lm14}\myLabel{lm14}\relax  If $ {\mathfrak g}\left(A\right) $ is a regular Kac-Moody superalgebra of finite
growth, then $ \Pi_{0} $ is finite. Moreover, $ |\Pi_{0}|\leq2n $.

\end{lemma}

\begin{proof}
Since all odd reflections involved are regular, $
a_{\alpha\alpha}=2 $ for any $ \alpha\in\Pi_{0} $. For any finite
subset $S \subset \Pi_{0}$ the  corresponding matrix $B$ has
finite growth, therefore it consists of affine and finite blocks.
The rank of $B$ is not bigger than $n$, on the other hand it is
bounded by the size of $B$ minus the number of affine blocks. That
implies $ |\Pi_{0}|\leq2n $.\end{proof}

\section{Regular Case}

We have the following results from [1].

\begin{lemma} If the matrix $A$ is symmetrizable, irreducible and has a
  zero on the diagonal,
then $\mathfrak {g} (A)$ is regular Kac-Moody if and only if $\mathfrak{g}(A)$ has
finite growth.
\end{lemma}

\begin{lemma}
Let $A$ be nonsymmetrizable with a zero on the diagonal and
$\mathfrak g(A)$ be regular Kac-Moody. Then $\mathfrak g(A)$ has
finite growth if and only if it is either $q(n)^{(2)}$ or
$S(1,2,a)$ for non-integer $a$.  $S(1,2,a)$ is isomorphic to
$S(1,2,b)$ iff $a \pm b \in {\mathbb Z}$. See diagrams in table
8.1.
\end{lemma}

\section{Non-regular elemental case}

Here we classify all finite growth Lie superalgebras $
{\mathfrak{g}\left(A\right)} $ where $ A $ is an irreducible,
elemental and $ {\mathfrak{g}\left(A\right)} $ is not regular
Kac-Moody. It is useful to describe $A$ by the corresponding Dynkin
diagram, see [2,10]. Here we use matrix diagrams, but still follow
the same labeling conventions for the vertices.
\begin{equation}
\begin{tabular}{|c|c|c|c|}
\hline
$\mathfrak{g}(A)$ & $A$ & $p(1)$ & Dynkin diagram \\
\hline
$A_1$ & (2) & 0 & $\bigcirc$ \\
$B(0,1)$ & (2) & 1 & \Huge$\bullet$\normalsize \\
$A(0,0)$ & (0) & 1 & $\bigotimes$ \\
\hline
\end{tabular}
\end{equation}

We assume the the matrix has only 0 or 2 on the diagonal and call
$A$ {\it admissible} if any matrix obtained from $A$ by a sequence
of odd reflections satisfies the conditions of Lemma ~\ref {lm23}.
We join vertex $i$ to vertex $j$ be an arrow if $a_{ij} \neq 0$ and
we write the number $a_{ij}$ on this arrow.

We proceed by induction on the number of vertices.  We say that $\Gamma'$ is a subdiagram of $\Gamma$ if $A'$ is
a principal submatrix of $A$.  We observe that if a diagram has only two vertices and one of them is singular,
then the diagram is non-elemental.  So our classification begins with $n=3$ vertices.

\begin{lemma} \label{lm6}\myLabel{lm6}\relax  Let $ A $ be an admissible $ 2\times2 $-matrix such that
$ a_{11}=0 $, $ a_{22}=2 $, and $a_{12} \neq 0$. If $ p\left(2\right)=0 $, then $ a_{21}=-1 $ or $ -2 $.
If $ p\left(2\right)=1 $, then $ a_{21}=-2 $.
\end{lemma}

\begin{proof}
We may assume $a_{12}=1$ without loss of generality.  Also, we know $a_{21} \neq 0$.  Now by reflecting at the
isotropic vertex $v_{1}$, we have
\begin{equation*}
\left(\begin{array}{cc}
0 & 1 \\
a_{21} & 2 \\
\end{array} \right)
\text{ $\overrightarrow{r_{1}}$ }
\left(\begin{array}{cc}
0 & 1 \\
-a_{21} & 2(a_{21}+1) \\
\end{array} \right).
\end{equation*}
If $a_{21} \neq -1$, then after rescaling the new matrix we have
\begin{equation*}
\left(\begin{array}{cc}
0 & 1 \\
\frac{-a_{21}}{a_{21}+1} & 2 \\
\end{array} \right).
\end{equation*}
Then $a_{21}, \frac{-a_{21}}{a_{21}+1} \in {\mathbb Z}_{-}$
implies that $a_{21}=-2$. Hence, $a_{21} \in \{-1,-2\}$.  Since
$a_{21} \in 2^{p\left(i\right)}{\mathbb Z}_{-}$, the result
follows.
\end{proof}

\begin{lemma} \label{lm7}\myLabel{lm7}\relax If $a_{31} \neq 0$, then the following matrix is not admissible:
\begin{equation*}A=
\left(\begin{array}{ccc}
a_{11} & 0 & a_{13}\\
0 & 0 & 1 \\
a_{31} & -2 & 2 \\
\end{array}\right).
\end{equation*} \end{lemma}

\begin{proof}
By reflecting at the isotropic vertex $v_{2}$ and then rescaling, we have
\begin{equation*}A'=
\left(\begin{array}{ccc}
a_{11} & 0 & a_{13}\\
0 & 0 & 1 \\
-a_{31} & -2 & 2 \\
\end{array}\right).
\end{equation*}
For matrix $A$ and $A'$ to be admissible, we must have that $a_{31}<0$ and $-a_{31} < 0$.
But, this is a contradiction.
\end{proof}

\begin{theorem}\label{th500}\myLabel{th500\relax} Every admissible
  diagram with 3 vertices which contains a singular vertex and
  defines an algebra of finite growth is isomorphic to $S(1,2,0)$
 or $D(2,1,0)$.
\end{theorem}

\begin{proof}
Let $ \Pi=\left\{\alpha_{1},\alpha_{2},\alpha_{3}\right\} $ and
assume that $ \alpha_{1} $ is singular. Then $p(\alpha_{1})=1$.
Without loss of generality we have $ a_{11}=a_{12}=0 $, $ a_{21}\neq
0,a_{13}=1 $.
We have these assumptions for $ A $ in the
following 3 lemmas.

\begin{lemma} \label{lm41}\myLabel{lm41}\relax  $ p(\alpha_{2})=0 $, and therefore $ a_{22}=2 $.
\end{lemma}

\begin{proof} Consider the singular reflection $ r_{1} $, and
let $ A' $ be the new Cartan matrix,
$\alpha'_{i}=r_{1}\left(\alpha_{i}\right) $. Then $\alpha_{2}'=
\alpha_{1}+\alpha_{2}$, $h'_{2}=a_{21}h_{1} $ and  $a'_{23}=a_{21}a_{13} \neq 0 $,
$ a'_{22}=0 $. This implies $p(\alpha_{2}')=1$.
Since $p(\alpha_{2}')=p(\alpha_{1})+p(\alpha_{2}) (\text{mod }2)$,
we conclude $p(\alpha_{2})=0$ and hence $a_{22}=2$.
\end{proof}

\begin{lemma} \label{lm42}\myLabel{lm42}\relax  If $ a_{33}=0 $, then $ A $ equals
$
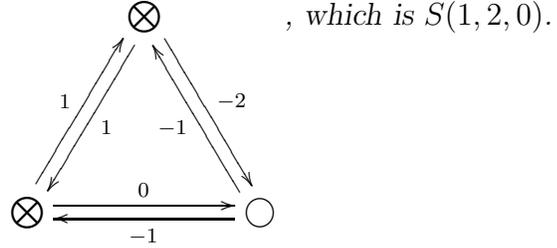
\xymatrix{& \bigotimes \ar@<.5ex>[rdd]^{-2} \ar@<.5ex>[ldd]^{1} & \\
& & \\ \bigotimes \ar@<.5ex>[ruu]^{1} \ar@<.5ex>[rr]^{0} & &
\bigcirc \ar@<.5ex>[ll]^{-1} \ar@<.5ex>[luu]^{-1}}$, which is $S(1,2,0)$.
\end{lemma}

\begin{proof} Note that by Lemma~\ref{lm41}, $ a_{31}\not=0 $ and therefore without
loss of generality we assume that $ a_{31}=1 $. By reflecting at $v_{1}$ we obtain
the matrix $ A' $
\begin{equation*}
\left( \begin{array}{rcr}
0 & 0 & 1 \\
0 & 0 & 1 \\
-1 & 1+a_{32} & 2 \\
\end{array} \right)
\end{equation*}
Since $ {\mathfrak g}\left(A'\right) $ must have finite growth
$a'_{32}=-1 $ or $ -2 $. By Lemma~\ref{lm7}, $ a'_{32} \neq -2 $.
Thus, $1+a_{32}=-1 $ and $ a_{32}=-2 $.  This implies $a_{23} \neq
0$. By considering the odd reflection $r_{3}$ we have the even
roots $ \alpha_{1}+\alpha_{3},\alpha_{2} \in \Pi_0$.  The subalgebra with
the corresponding positive generators  $X'_{1}=[X_{1},X_{3}]$ and
$X'_{2}=X_{2}$
must have finite growth.  So by Corollary~\ref{cor1}, the Lie algebra
$ {\mathfrak g}\left(B\right) $ must have finite growth, where we calculate that
\begin{equation*}B=
\left( \begin{array}{cr}
2 & -2 \\
a_{21}+a_{23} & 2 \\
\end{array} \right).
\end{equation*}
Recall that $ a_{21},a_{23}\in{\mathbb Z}_{<0} $, so
$a_{21}+a_{23}\leq-2 $.
Thus, $a_{21}+a_{23}=-2 $ and $ a_{21}=a_{23}=-1 $.\end{proof}

\begin{lemma} \label{lm43}\myLabel{lm43}\relax  If $ a_{33}=2 $, then $ A $ equals
$\xymatrix{ \bigcirc \ar@<.5ex>[r]^{-1} & \bigotimes
\ar@<.5ex>[l]^{1} \ar@<.5ex>[r]^{0} & \bigcirc
\ar@<.5ex>[l]^{-1}}$, which is $D(2,1,0)$.
\end{lemma}

\begin{proof} Then $a_{31},a_{32} \in {\mathbb Z}_{-}$.
Now $ a_{13}\not=0 $ and $ a_{33}=2 $ imply $ a_{31}\not=0 $. Then by Lemma~\ref{lm6} $
a_{31}=-1 $ or $ -2 $. Suppose $ a_{31}=-2 $. Then after an application
of $ r_{1} $ we obtain the matrix
\begin{equation*} \left(
\begin{array}{rcr}
0 & 0 & 1 \\
0 & 0 & 1 \\
-2 & 2-a_{32} & 2 \\
\end{array} \right)
\end{equation*}
This implies  $ 2-a_{32}\in{\mathbb Z}_{<0} $, but this is
impossible since $a_{32}\in{\mathbb Z}_{<0} $.  Hence $ a_{31}=-1$.
By reflecting at $v_{1}$ and then at $v_{3}$ we obtain
\begin{equation*}
\text{ $\overrightarrow{r_{1}}$ }
\left( \begin{array}{rcr}
0 & 0 & 1 \\
0 & 0 & 1 \\
1 & a_{32}-1 & 0 \\
\end{array} \right)
\text{ $\overrightarrow{r_{3}}$ }
 \left( \begin{array}{rcr}
2 & a_{32} & -1 \\
\frac{a_{32}}{a_{32}-1} & 2 & -1 \\
1 & a_{32}-1 & 0 \\
\end{array} \right).
\end{equation*}
Then $a_{32}, \frac{a_{32}}{a_{32}-1} \in {\mathbb Z}_{-}$ implies $a_{32}=0$.
Thus $ a_{23}=0 $.

Now we need to find $ a_{21} $.  Observe that $\alpha_{2}, 2\alpha_{1}+\alpha_{2}+\alpha_{3} \in \Pi_0$, and
and the subalgebra generated by $X'_{1}=[[X_{1},X_{2}],[X_{1},X_{3}]]$ and $X'_{2}=X_{2}$
must have finite growth.  By Corollary~\ref{cor1}, the Lie algebra
$ {\mathfrak g}\left(B\right) $ must have finite growth, where we calculate that
\begin{equation*}B=
\left( \begin{array}{cr}
2 & 0 \\
2a_{21}+2 & 2 \\
\end{array} \right).
\end{equation*}
We can conclude that $ a_{21}=-1$.\end{proof}
\end{proof}

\begin{lemma} \label{lm100}\myLabel{lm100}\relax
Suppose that $A$ is a $3\times3$ generalized Cartan matrix, but    $
{\mathfrak{g}\left(A\right)} $ is not regular Kac-Moody. If $
{\mathfrak{g}\left(A\right)} $ has finite growth, then $A$ is one of
the following diagrams:
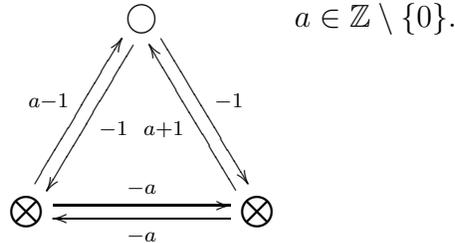
\begin{equation*}
\xymatrix{& \bigcirc \ar@<.5ex>[rdd]^{-1} \ar@<.5ex>[ldd]^{-1} & \\
& & \\ \bigotimes \ar@<.5ex>[ruu]^{a-1} \ar@<.5ex>[rr]^{-a} & &
\bigotimes \ar@<.5ex>[ll]^{-a} \ar@<.5ex>[luu]^{a+1}} \text{  } a
\in \mathbb{Z}\setminus\{0\}.
\end{equation*}
\end{lemma}

\begin{proof}
If $A$ does not have a singular root and $
{\mathfrak{g}\left(A\right)} $ is not regular Kac-Moody, then $A$
must reflect by a sequence of regular reflections to some $A'$ which
contains a singular root. Then $A'$ must contain at least two
isotropic roots, one singular and one regular. Hence $A'$ is
isomorphic to $S(1,2,0)$.
 The claim now immediately follows from the fact that the set $S(1,2,a)$
 for integer $a$ is closed under
regular odd reflections and contains the $S(1,2,0)$ diagram with a singular vertex.
\end{proof}

\begin{theorem}
Every admissible diagram with 4 vertices which contains a singular
vertex and defines an algebra of finite growth is isomorphic to
$\widehat{D}(2,1,0)$.  An admissible diagram with $n \geq 5$
vertices which contains a singular vertex does not define an algebra
of finite growth.
\end{theorem}
\begin{proof}
Every admissible non-regular contragredient Lie superalgebra of
finite growth with $n>3$ must contain an admissible 3 vertex
subdiagram which is non-regular. So it suffices to find all
admissible extensions of $S(1,2,0)$ and $D(2,1,0)$.

\begin{lemma}
$S(1,2,0)$ is not extendable.
\end{lemma}

\begin{proof}
Suppose $A$ is an admissible matrix, such that the submatrix given by excluding the vertex $v_{4}$ is $S(1,2,0)$,
\begin{equation*}
\left( \begin{array}{rrrr}
0 & 1 & 0 & a_{14} \\
1 & 0 & -2 & a_{24} \\
-1 & -1 & 2 & a_{34} \\
a_{41} & a_{42} & a_{43} & a_{44} \\
\end{array} \right).
\end{equation*}

Case 1:  $p(4)=0$ \\
Then $a_{44}=2$, and $a_{41},a_{42},a_{43} \in {\mathbb Z}_{-}$.  The subalgebra corresponding to the even roots
$\alpha_{1}+\alpha_{2},\alpha_{3},\alpha_{4} \in \Pi_0$, with positive generators $[X_{1},X_{2}]$ ,$X_{3}$ , $X_{4}$
must have finite growth.  So by Corollary~\ref{cor1}, the Lie algebra $ {\mathfrak{g}\left(B\right)} $
must have finite growth, where we calculate that
\begin{equation*}B=
\left( \begin{array}{ccc}
2 & -2 & a_{14}+a_{24}\\
-2 & 2 & a_{34}\\
a_{41}+a_{42} & a_{43} & 2\\
\end{array} \right).
\end{equation*}
Hence, $a_{34}=a_{43}=0$, $a_{41}+a_{42}=0$, and
$a_{14}+a_{24}=0$. Since $a_{41},a_{42} \leq 0$ this implies
$a_{41}=a_{42}=0$.  It follows that $a_{14}=a_{24}=0$.  Hence,
case 1 is not admissible.

Case 2:  $p(4)=1$ and $a_{44}=2$ \\
Then $a_{41},a_{42},a_{43} \in {\mathbb Z}_{-}$.  The subalgebra corresponding to the even roots
$\alpha_{1}+\alpha_{2},\alpha_{3},2\alpha_{4} \in \Pi_0$, with positive generators $[X_{1},X_{2}]$, $X_{3}$, $[X_{4},X_{4}]$.
must have finite growth.  So by Corollary~\ref{cor1}, the Lie algebra $ {\mathfrak{g}\left(B\right)} $
must have finite growth, where we calculate that
\begin{equation*}B=
\left( \begin{array}{ccc}
2 & -2 & 2(a_{14}+a_{24})\\
-2 & 2 & 2a_{34}\\
\frac{1}{2}(a_{41}+a_{42}) & \frac{1}{2}a_{43} & 2\\
\end{array} \right).
\end{equation*}
By the same argument as in case 1, we conclude that case 2 is not admissible.

Case 3: $p(4)=1$ and $a_{44}=0$, with $a_{24} \neq 0$ \\
Then $a_{42} \neq 0$ by Lemma~\ref{lm41}, and WLOG $a_{42}=1$. The
subalgebra corresponding to the even roots
$\alpha_{1}+\alpha_{2},\alpha_{3},\alpha_{2}+\alpha_{4} \in \Pi_0$,
with positive generators $[X_{1},X_{2}]$, $X_{3}$, $[X_{4},X_{2}]$
must have finite growth.  So by Corollary~\ref{cor1}, the Lie
algebra $ {\mathfrak{g}\left(B\right)} $ must have finite growth,
where we calculate that
\begin{equation*}B=
\left( \begin{array}{ccc}
2 & -2 & 1+a_{14}+a_{24} \\
-2 & 2 & a_{34}-1 \\
1+a_{24}(1+a_{41}) & -2+a_{24}a_{43} & 2a_{24}\\
\end{array} \right).
\end{equation*}
Hence, $a_{34}-1=0$ implying $a_{34}=1$.
But, $a_{34} \in {\mathbb Z}_{-}$, which is a contradiction.
Hence, case 3 is not admissible.

Case 4:  $p(4)=1$ and $a_{44}=0$, with $a_{24}=0$ \\
Then $a_{42}=0$ by Lemma~\ref{lm41}.  For the submatrix excluding
$v_{3}$ to be admissible, either $a_{14}=0$ or $a_{14}=-1$. Suppose
$a_{14}=0$. Then $a_{41}=0$.  Since A is elemental, $a_{43} \neq 0$.
Then $a_{34} \neq 0$.  For the submatrix excluding $v_{1}$ to be
admissible $a_{34} \in \{-1,-2\}$ by Lemma~\ref{lm6}. By
Lemma~\ref{lm7}, $a_{34} \neq -2$.  Hence, $a_{34}=-1$ and WLOG
$a_{43}=1$.  Then by reflecting at $v_{4}$ we have $A'$ is
\begin{equation*} \left(
\begin{array}{rrrr}
0 & 1 & 0 & 0 \\
1 & 0 & -2 & 0 \\
-1 & -1 & 0 & 1 \\
0 & 0 & 1 & 0 \\
\end{array} \right).
\end{equation*}
But, the submatrix excluding $v_{4}$ of $A'$ is not admissible.
Hence, $a_{14}=-1$. Then the submatrix excluding $v_{2}$ of $A$ must
be $S(1,2,0)$ by Lemma~\ref{lm42}, which implies $a_{34}=-1$,
$a_{41}=1$ and $a_{43}=-2$.  Then by reflecting at $v_{2}$ we have
\begin{equation*}
 \left(
\begin{array}{rrrr}
0 & 1 & 0 & -1 \\
1 & 0 & -2 & 0 \\
-1 & -1 & 2 & -1 \\
1 & 0 & -2 & 0 \\
\end{array} \right)
\text{ $\overrightarrow{r_{2}}$ }
 \left(
\begin{array}{rrrr}
2 & -1 & -1 & -1 \\
1 & 0 & -2 & 0 \\
3 & -2 & 0 & 2 \\
1 & 0 & -2 & 0 \\
\end{array} \right).
\end{equation*}
The submatrix excluding $v_{2}$ of $A'$ is not admissible by the
three vertex classification. Hence, case 4 is not admissible.
Therefore, $S(1,2,0)$ is not extendable.
\end{proof}

\begin{lemma}
The only admissible extension of $D(2,1,0)$ is the diagram $\widehat{D}(2,1,0)$
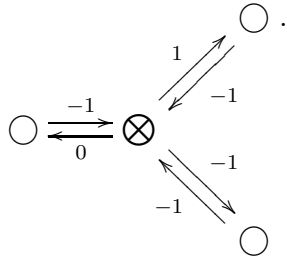
\begin{equation*}
\xymatrix{
     & & \bigcirc  \ar@<.5ex>[ld]^{-1}  \\
     \bigcirc \ar@<.5ex>[r]^{-1} & \bigotimes \ar@<.5ex>[l]^{0} \ar@<.5ex>[ru]^{1}
     \ar@<.5ex>[rd]^{-1} & \\
     & & \bigcirc  \ar@<.5ex>[lu]^{-1}  }.
\end{equation*}  The diagram $\widehat{D}(2,1,0)$ is not extendable.
\end{lemma}

\begin{proof}
Suppose $A$ is an admissible matrix, such that the submatrix given by excluding the vertex $v_{4}$ is $D(2,1,0)$,
\begin{equation*}
\left( \begin{array}{rrrr}
2 & -1 & 0 & a_{14} \\
1 & 0 & 0 & a_{24} \\
0 & -1 & 2 & a_{34} \\
a_{41} & a_{42} & a_{43} & a_{44} \\
\end{array} \right).
\end{equation*}

Case 1: $a_{24}=0$, $a_{42} \neq 0$ \\
Then the for the submatrix given by excluding the vertex $v_{3}$ to be admissible it must be $D(2,1,0)$.  Thus
$a_{14}=a_{41}=0$, $a_{42}=-1$ and $p(4)=0$.  By reflecting at $v_{2}$ we obtain the matrix $A'$
\begin{equation*}
\left( \begin{array}{rrrr}
0 & 1 & -1 & -1 \\
1 & 0 & 0 & 0 \\
1 & 0 & 0 & 0 \\
1 & 0 & 0 & 0 \\
\end{array} \right).
\end{equation*}
But, the submatrix given by excluding $v_{2}$ of $A'$ is not
admissible.

Case 2:  $a_{24}=0$, $a_{42}=0$ \\
Then by reflecting at $v_{2}$ we obtain $A'$
\begin{equation*}
\left( \begin{array}{rrrr}
0 & 1 & -1 & a_{14} \\
1 & 0 & 0 & 0 \\
1 & 0 & 0 & 0 \\
a_{41} & 0 & a_{43} & a_{44} \\
\end{array} \right).
\end{equation*}
Consider the submatrix of $A'$  given by excluding the vertex
$v_{2}$.  Since $S(1,2,0)$ is not extendable, we have that
$a_{43}=0$.  Now consider the original matrix $A$.  If $a_{14}=0$
then $a_{41}=0$ since $a_{11}=2$. Since $A$ is elemental this
implies $a_{44}=2$. Then $a_{34}=0$ and so $A$ is decomposable. So
we have that $a_{14} \neq 0$.  Since $a_{14} \in \mathbb{Z}_{-}$,
the submatrix of $A'$ given by excluding the vertex $v_{2}$ is not
admissible.

Case 3:  $a_{24} \neq 0$ \\
Then the for the submatrix given by excluding the vertex $v_{1}$ to
be admissible it must be $D(2,1,0)$ since $S(1,2,0)$ is not
extendable. Then $a_{34}=a_{43}=0$, $a_{42}=-1$, $a_{44}=2$ and
$p(4)=0$.  For the submatrix given by excluding $v_{3}$ to have
finite growth, either $a_{24}=a_{14}=a_{41}=-1$; or
$a_{14}=a_{41}=0$.  So first suppose $a_{24}=a_{14}=a_{41}=-1$. Then
by reflecting at $v_{2}$ we obtain the matrix $A'$
\begin{equation*}
\left( \begin{array}{rrrr}
0 & 1 & -1 & -1 \\
1 & 0 & 0 & -1 \\
-1 & 0 & 0 & 1 \\
1 & -1 & 1 & 0\\
\end{array} \right).
\end{equation*}
But, the submatrix given by excluding $v_{2}$ is not admissible.
Thus $a_{14}=a_{41}=0$.  Then by reflecting at $v_{2}$ we obtain the
matrix $A'$
\begin{equation*}
\left( \begin{array}{cccc}
0 & 1 & -1 & -1-a_{24} \\
1 & 0 & 0 & a_{24} \\
-1 & 0 & 0 & -a_{24} \\
-1-a_{24} & a_{24} & -a_{24} & 0 \\
\end{array} \right).
\end{equation*}
For the submatrix excluding $v_{2}$ of $A'$ to have finite growth,
$-2-2a_{24}=0$.  So $a_{24}=-1$. This case is admissible, and we
call it $\widehat{D}(2,1,0)$. Now we only need to show that
$\widehat{D}(2,1,0)$ does not extend to an admissible 5 vertex
diagram.  So suppose that $A$ is a Cartan matrix of an admissible 5
vertex extension of $\widehat{D}(2,1,0)$. Then each 4 vertex
subdiagram containing $D(2,1,0)$ must be $\widehat{D}(2,1,0)$.  This
reduces us to the following diagram
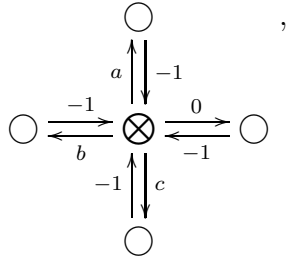
\begin{equation*}
\xymatrix{ & \bigcirc \ar@<.5ex>[d]^{-1}   & \\
     \bigcirc \ar@<.5ex>[r]^{-1}  &
     \bigotimes \ar@<.5ex>[u]^{a} \ar@<.5ex>[l]^{b} \ar@<.5ex>[d]^{c} \ar@<.5ex>[r]^{0}
     & \bigcirc \ar@<.5ex>[l]^{-1} \\
     & \bigcirc \ar@<.5ex>[u]^{-1} &  },
\end{equation*}
where $a=-b$, $b=-c$, and $c=-a$.  But, this is a contradiction since $a,b,c \neq 0$.  Hence, $\widehat{D}(2,1,0)$ is not extendable.
\end{proof}
For $n \geq 6$, an admissible diagram with a singular vertex must contain an admissible 5 vertex subdiagram with a
singular vertex.  Since there are none with 5 vertices, we have found all admissible diagrams containing a singular vertex.
\end{proof}

\begin{corollary}
Suppose the matrix $A$ is a generalized Cartan matrix, but
${\mathfrak g}(A)$ is not regular Kac-Moody.  If ${\mathfrak g}(A)$
has finite growth, then $A$ is one of the diagrams in
Lemma~\ref{lm100}.
\end{corollary}

\begin{proof}
This follows immediately from Lemma~\ref{lm100} and the fact that $\widehat{D}(2,1,0)$ does not contain a regular
isotropic vertex.
\end{proof}
\newpage
\section{Non-symmetrizable Contragredient Lie Superalgebras of Finite Growth (with elemental matrices)}
\begin{tabular}{|c|c|}
  \hline
& \\
  \bf{Algebra} & \bf{Dynkin diagrams} \hspace{1cm}  Table 8.1 \\
& \\
  \hline
& \\
$D(2,1,0)$ & $\xymatrix{\bigcirc \ar@<.5ex>[r]^{-1} & \bigotimes
\ar@<.5ex>[l]^{0} \ar@<.5ex>[r]^{1}
 & \bigcirc \ar@<.5ex>[l]^{-1} }$ \\
 & \\
 \hline
$\begin{array}{c} \\ \\ \\ \widehat{D}(2,1,0) \end{array}$ &
$\xymatrix{
     & & \bigcirc  \ar@<.5ex>[ld]^{-1}  \\
     \bigcirc \ar@<.5ex>[r]^{-1} & \bigotimes \ar@<.5ex>[l]^{0} \ar@<.5ex>[ru]^{1}
     \ar@<.5ex>[rd]^{-1} & \\
     & & \bigcirc  \ar@<.5ex>[lu]^{-1}  } $ \\
& \\
\hline
$\begin{array}{c} \\ \\ \\ S(1,2,0) \end{array}$ &
$\xymatrix{& \bigcirc \ar@<.5ex>[rdd]^{-1} \ar@<.5ex>[ldd]^{-1} & \\
& & \\ \bigotimes \ar@<.5ex>[ruu]^{a-1} \ar@<.5ex>[rr]^{-a} & &
\bigotimes \ar@<.5ex>[ll]^{-a} \ar@<.5ex>[luu]^{a+1}}$ \text{  } $a \in \mathbb{Z}$ \\
\hline
$\begin{array}{c} \\ \\ \\ S(1,2,a) \end{array}$ &
$\xymatrix{& \bigcirc \ar@<.5ex>[rdd]^{-1} \ar@<.5ex>[ldd]^{-1} & \\
& & \\ \bigotimes \ar@<.5ex>[ruu]^{a-1} \ar@<.5ex>[rr]^{-a} & &
\bigotimes \ar@<.5ex>[ll]^{-a} \ar@<.5ex>[luu]^{a+1}}$ \text{  } $a \not \in \mathbb{Z}$ \\
 \hline
& \\
$q(n)^{(2)}$ & $\begin{array}{c}\xymatrix{ &  \bullet \AW[ld]^{a}_{} \AW[rd]^{b}_{}  & \\
\bullet  \ar[r] &  \cdots \ar[l] \ar[r] & \bullet \ar[l]\\}\end{array}$
 \text{  } $\begin{array}{c} \text{There are n $\bullet$.} \\
\text{Each $\bullet$ is either $\bigcirc$ or $\bigotimes$.} \\
\text{An odd number of them are
$\bigotimes$.} \\
\text{If $\bullet$ is $\bigcirc$,
 then $a=b=-1$.} \\
\text{If $\bullet$ is $\bigotimes$, then $\frac{a}{b}=-1$.} \\
 \end{array}$
\\
& \\
\hline
\end{tabular}

\section{Description of non-symmetrizable contragredient superalgebras }

As follows from our classification there are the following
non-symmetrizable superalgebras: $ D\left(2,1;0\right) $, $
\widehat{D}\left(2,1;0\right) $, $
q\left(n\right)^{\left(2\right)} $ and $ S\left(1,2;a\right) $. \\

The Lie superalgebra $ D\left(2,1;0\right) $ is the degenerate
member of the family $ D\left(2,1;\alpha\right) $ when $ \alpha=0
$, described in $ \left[2\right] $. $ D\left(2,1;0\right) $ is not
simple, but has the ideal isomorphic to $ p sl\left(2|2\right) $.
It is not difficult to check that
$D\left(2,1;0\right)\cong\operatorname{Der}\left(psl\left(2|2\right)\right) $, more precisely,
$ D\left(2,1;0\right) $
is isomorphic to a semidirect sum of $ sl\left(2\right) $ and $ p
sl\left(2|2\right) $, defined as follows. Let $ H,E $ and $ F $
denote the standard basis of $ sl\left(2\right) $, $ {\mathfrak l}
$ denote the Lie superalgebra $ p sl\left(2|2\right) $. Consider
the standard $ {\mathbb Z} $-grading $ {\mathfrak l}={\mathfrak
l}_{-1}\oplus{\mathfrak l}_{0}\oplus{\mathfrak l}_{1} $, where $
{\mathfrak l}_{0}=sl\left(2\right)\oplus sl\left(2\right) $, $
{\mathfrak l}_{1} $ and $ {\mathfrak l}_{-1} $ are isomorphic $
{\mathfrak l}_{0} $-modules. For every $ g\in{\mathfrak l}_{i} $
define $ \left[H,g\right]=ig $, furthermore let
\begin{equation}
\left[E,{\mathfrak l}_{0}\right]=\left[F,{\mathfrak
l}_{0}\right]=\left[E,{\mathfrak l}_{1}\right]=\left[F,{\mathfrak
l}_{-1}\right]=0, \notag\end{equation}
\begin{equation}
\left[E,x\right]=\gamma\left(x\right)\text{, }x\in{\mathfrak
l}_{-1}\text{,
}\left[F,y\right]=\gamma^{-1}\left(y\right),y\in{\mathfrak l}_{1},
\notag\end{equation} where $ \gamma:{\mathfrak l}_{-1} \to
{\mathfrak l}_{1} $ is an isomorphism of $ {\mathfrak l}_{0}
$-modules. To identify the Chevalley generators let $X_1=F, Y_1=E$
and let $X_2,X_3$ be the simple roots of $\mathfrak{l}$.

{\em Description of\/} $ \widehat{D}\left(2,1,0\right) $. Consider
the affine Lie superalgebra $ \widehat{{\mathfrak l}} $ isomorphic to
\begin{equation}
{\mathfrak l}\otimes{\mathbb C}\left[t,t^{-1}\right]\oplus{\mathbb
C}c\oplus{\mathbb C}d, \notag\end{equation} where $ c $ belongs to
the center and the commutator is determined by the formulas
\begin{equation}
\left[x\otimes t^{m},y\otimes t^{n}\right]=\left[x,y\right]\otimes
t^{m+n}+\left(x,y\right)m\delta_{m,-n}c, \notag\end{equation}
\begin{equation}
\left[d,x\otimes t^{m}\right]=mx\otimes t^{m},
\notag\end{equation} where $ x,y\in{\mathfrak l} $. Define the
semidirect product of $ sl\left(2\right) $ and $
\widehat{{\mathfrak l}} $ by the formula
\begin{equation}
\left[g,x\otimes t^{m}+a c+b d\right]=\left[g,x\right]\otimes
t^{m} \notag\end{equation} for all $ g\in sl\left(2\right) $, $
x\in{\mathfrak l} $, $ a,b\in{\mathbb C} $. The Lie superalgebra $
\widehat{D}\left(2,1;0\right) $ is isomorphic to this semidirect
product. To identify Chevalley generators let again $X_1=F, Y_1=E$ and
$X_2,X_3, X_4$ be the simple roots $\widehat{\mathfrak l}$.

{\em Description of\/} $ q\left(n\right)^{\left(2\right)} $. The Lie
superalgebra $ q\left(n\right)^{\left(2\right)} $ is an extension of
the loop algebra of the simple Lie superalgebra $ {\mathfrak
t}=q\left(n\right) $ twisted by the involutive automorphism $ \phi $
such that $ \phi_{|{\mathfrak t}_{0}}=\operatorname{id} $, $
\phi_{|{\mathfrak t}_{1}}=-id. $ The Lie superalgebra $
q\left(n\right)^{\left(2\right)} $ is isomorphic to
\begin{equation}
{\mathfrak t}_{0}\otimes{\mathbb
C}\left[t^{2},t^{-2}\right]\oplus{\mathfrak t}_{1}\otimes t{\mathbb
C}\left[t^{2},t^{-2}\right]\oplus{\mathbb C} c\oplus{\mathbb C}d
\notag\end{equation} with $ d=t\frac{\partial}{\partial t} $ and $ c
$ is the central element. For any $ x,y\in{\mathfrak t} $, the
commutator is defined by the formula
\begin{equation}
\left[x\otimes t^{m},y\otimes t^{n}\right]=\left[x,y\right]\otimes
t^{m+n}+\delta_{m,-n}\left(1-\left(-1\right)^{m}\right)\operatorname{tr}
\left(xy\right) c. \notag\end{equation}

{\em Description of\/} $ S\left(1,2,a\right) $. The Lie superalgebra
$ S\left(1,2,a\right) $ appears in the list of conformal Lie
superalgebras by Kac and Van de Leur $ \left[4\right] $. Consider
commutative associative Lie superalgebra $ R={\mathbb
C}\left[t,t^{-1},\xi_{1},\xi_{2}\right] $ with even generator $ t $
and two odd generators $ \xi_{1},\xi_{2} $. By $ {\mathcal
W}\left(1,2\right) $ we denote the Lie superalgebra of derivations
of $ R $, $ i.e. $ all linear maps $ d:R \to R $ such that
\begin{equation}
d\left(fg\right)=d\left(f\right)
g+\left(-1\right)^{p\left(d\right)p\left(f\right)}f
d\left(g\right). \notag\end{equation} An element $ d\in
W\left(1,2\right) $ can be written as
\begin{equation}
d=f\frac{\partial}{\partial
t}+f_{1}\frac{\partial}{\partial\xi_{1}}+f_{2}\frac{\partial}{\partial\xi_{2}}
\notag\end{equation} for some $ f,f_{1},f_{2}\in R $. It is easy
to see that the subset of all $ d\in{\mathcal W}\left(1,2\right) $
satisfying the condition
\begin{equation}
a t^{-1}f+\frac{\partial f}{\partial
t}=\left(-1\right)^{p\left(d\right)}\left(\frac{\partial
f_{1}}{\partial\xi_{1}}+\frac{\partial
f_{2}}{\partial\xi_{2}}\right) \notag\end{equation} form a
subalgebra of $ {\mathcal W}\left(1,2\right) $, which we denote by
$ {\mathcal S}_{a} $. One can check that $ {\mathcal S}_{a} $ is
simple if $ a\notin{\mathbb Z} $. If $ a\in{\mathbb Z} $, then the
ideal $ {\mathcal S}'_{a}=\left[{\mathcal S}_{a},{\mathcal
S}_{a}\right] $ is simple and has codimension 1, more precisely
\begin{equation}
{\mathcal S}_{a}={\mathbb
C}\xi_{1}\xi_{2}t^{-a}\frac{\partial}{\partial t}\oplus{\mathcal
S}_{a}'. \notag\end{equation} Set
\begin{equation}
X_{1}=-\frac{\partial}{\partial\xi_{1}}\text{,
}X_{2}=-a \xi_{1}\xi_{2}t^{-1}\frac{\partial}{\partial\xi_{1}}+{\xi_{2}}\frac{\partial}{\partial
t}\text{, }X_{3}=\xi_{1}\frac{\partial}{\partial\xi_{2}},
\notag\end{equation}
\begin{equation}
Y_{1}=(a+1)\xi_{1}\xi_{2}\frac{\partial}{\partial\xi_{2}}+\xi_{1}{t}\frac{\partial}{\partial
t}\text{, }Y_{2}=t\frac{\partial}{\partial\xi_{2}}\text{,
}Y_{3}=\xi_{2}\frac{\partial}{\partial\xi_{1}},
\notag\end{equation}
\begin{equation}
h_{1}=-(a+1)\xi_{2}\frac{\partial}{\partial\xi_{2}}-{t}\frac{\partial}{\partial
t}\text{,
}h_{2}=a \xi_{1}\frac{\partial}{\partial\xi_{1}}+{\xi_{2}}\frac{\partial}{\partial\xi_{2}}+{t}\frac{\partial}{\partial
t}\text{,
}h_{3}=\xi_{1}\frac{\partial}{\partial\xi_{1}}-\xi_{2}\frac{\partial}{\partial\xi_{2}}.
\notag\end{equation}

Then $ X_{i},Y_{i},h_{i} $, $ i=1,2,3 $, generate $ {\mathcal
S}_{a} $ or $ {\mathcal S}_{a}' $ if $ a\in{\mathbb Z} $. They
satisfy the relations~\eqref{equ1} with Cartan matrix $
S\left(1,2,a\right) $. Hence the
contragredient Lie superalgebra $ S\left(1,2,a\right) $ can be
obtained from $ {\mathcal S}_{a}\left({\mathcal S}_{a}'\right) $
by adding the element $ d=t\frac{\partial}{\partial t} $ and
taking a central extension. The formula for this central extension
can be found in [4]. \\
The fact that the algebras in the table 8.1 have finite growth follows directly from their descriptions.

\section{Non-elemental case }

The description of superalgebras with non-elemental Cartan
matrices can be easily obtained from the following

\begin{lemma} \label{lm22}\myLabel{lm22}\relax  Let $ A $ be indecomposable with $ a_{1 i}=0 $ for all $ i\in I $, and let
$ J=I\backslash\left\{1\right\} $. Then $ {\mathfrak
g}\left(A\right) $ has a $ {\mathbb Z} $-grading $ {\mathfrak
g}={\mathfrak g}_{-1}\oplus{\mathfrak g}_{0}\oplus{\mathfrak g}_{1}
$ such that $ {\mathfrak g}_{0}\cong{\mathfrak
g}\left(A_{J}\right)\oplus{\mathbb C}h_{1}\oplus{\mathbb C}d $,
where $ d $ is a non-zero vector in $ \cap_{i\in
J\backslash\left\{1\right\}}\operatorname{Ker} \alpha_{i} $ not
proportional to $ h_{1} $, $ {\mathfrak g}_{-1} $ is an irreducible
$ {\mathfrak g}_{0} $-module with highest weight $ -\alpha_{1} $, $
{\mathfrak g}_{1} $ is the $ {\mathbb Z} $-graded dual of $
{\mathfrak g}_{-1} $ (which is a lowest weight module with lowest
weight $ \alpha_{1} $) and the commutator between $ {\mathfrak
g}_{-1} $ and $ {\mathfrak g}_{1} $ is defined by the formula
\begin{equation}
\left[x,y\right]=\left<x,y\right>h_{1},
\notag\end{equation}
here $ \left< x,y \right> $ denotes the natural pairing between $ x\in{\mathfrak g}_{-1} $ and $ y\in{\mathfrak g}_{1} $.

\end{lemma}

\begin{proof} First, we have to show that the algebra $ {\mathfrak g} $ is a quotient of
$ \bar{{\mathfrak g}}\left(A\right) $. Indeed, let $ X_{i} $, $ Y_{i} $ with $ i\geq2 $ be the generators of $ {\mathfrak g}_{0} $, $ X_{1} $ be the lowest
vector of $ {\mathfrak g}_{1} $ and, $ Y_{1} $ be the highest vector of $ {\mathfrak g}_{-1} $, $ {\mathfrak h} $ be the direct sum of
the Cartan subalgebra of $ {\mathfrak g}_{0} $, $ {\mathbb C}h_{1} $ and $ {\mathbb C}d $. It is easy to see that
$ X_{1},\dots ,X_{n},Y_{1},\dots ,Y_{n} $ and
$ h_{1},\dots ,h_{n} $ satisfy the relations~\eqref{equ1}. Now we have to check that $ {\mathfrak g} $ does not
have
non-zero ideals, which intersect $ {\mathfrak h} $ trivially. If
${\mathcal I}$ is such ideal, then
$ {\mathcal I}\cap{\mathfrak g}_{0}=\left\{0\right\} $, and therefore $ {\mathcal I}=\left({\mathcal I}\cap{\mathfrak g}_{1}\right)\oplus\left({\mathcal I}\cap{\mathfrak g}_{-1}\right) $. Since $ {\mathfrak g}_{1} $ is an irreducible
$ {\mathfrak g}_{0} $-module, $ {\mathcal I}\cap{\mathfrak g}_{1}=\left\{0\right\} $ or $ {\mathcal I}={\mathfrak g}_{1} $. If $ {\mathcal I}={\mathfrak g}_{1} $, then $ X_{1}\in{\mathcal I} $, and $ h_{1}\in{\mathcal I} $. This is
impossible, hence $ {\mathcal I}\cap{\mathfrak g}_{1}=\left\{0\right\} $. In the same way $ {\mathcal I}\cap{\mathfrak g}_{-1}=\left\{0\right\} $. The lemma is
proven.\end{proof}

\begin{corollary} \label{cor221}\myLabel{cor221}\relax  Let $ A $ be an indecomposable Cartan matrix, $ I' $ be
the subset of all $ i\in I $ such that $ a_{ij}=0 $ for all $ j\in I $, $ A'=A_{I-I'} $, $ {\mathfrak g}'={\mathfrak g}\left(A'\right) $.
Then $ {\mathfrak g}\left(A\right) $ is a semi-direct sum of $ {\mathfrak g}' $ and the ideal $ {\mathfrak H}={\mathfrak g}^{-}\oplus{\mathfrak g}^{+}\oplus\cap_{i\in I-I'}\operatorname{Ker} \alpha_{i} $,
where $ {\mathfrak g}^{-} $ is a direct sum of $ |I'| $ irreducible highest
weight $ {\mathfrak g}' $-modules, $ {\mathfrak g}^{+} $ is a direct sum of $ |I'| $ irreducible lowest weight
$ {\mathfrak g}' $-modules, dual to $ {\mathfrak g}^{-} $, $ \left[{\mathfrak g}^{+},{\mathfrak g}^{+}\right]=\left[{\mathfrak g}^{-},{\mathfrak g}^{-}\right]=0 $, $ {\mathfrak g}^{+} $ and $ {\mathfrak g}^{-} $ generate a direct
sum of Heisenberg Lie superalgebras $ \oplus_{i\in I'}{\mathfrak H}_{i} $. Furthermore, $ {\mathfrak g}\left(A\right) $ has
finite growth if and
only if $ {\mathfrak g}' $ has finite growth and $ {\mathfrak g}^{\pm} $ have finite growth in the natural
$ {\mathbb Z} $-grading induced by the principal grading of $ {\mathfrak g}' $.

\end{corollary}

\begin{corollary} \label{cor223}\myLabel{cor223}\relax  Let $ {\mathfrak g}' $ be a contragredient Lie superalgebra with
Cartan matrix $ A' $ of size $ n $, $ V_{1},\dots ,V_{r} $ be irreducible highest weight
$ {\mathfrak g}' $-modules with
highest weight $ \lambda_{1},\dots ,\lambda_{r} $. Let $ V_{1}^{*},..,V_{r}^{*} $ denote the graded duals of
$ V_{1},\dots ,V_{r} $. Let $ {\mathfrak H} $ denote the vector space
\begin{equation}
{\mathfrak H}=V_{1}\oplus_{\dots }\oplus V_{r}\oplus V_{1}^{*}\oplus_{\dots }\oplus V_{r}^{*}\oplus{\mathbb C}^{2r}.
\notag\end{equation}
Choose a basis $ c_{1},\dots ,c_{r},d_{1},\dots ,d_{r} $ in $ {\mathbb C}^{2r} $ and define the Lie superalgebra
structure on $ {\mathfrak H} $ assuming that $ c_{i} $ lie in the center of $ {\mathfrak H} $ and the following
relations hold
\begin{equation}
\left[x,y\right]=\delta_{ij}\left<x,y\right>c_{i}\text{,
}\left[d_{k},x\right]=\delta_{ik}x\text{,
}\left[d_{k},y\right]=-\delta_{jk}y \notag\end{equation} for all $
x\in V_{i}^{*},y\in V_{j} $. Let $ {\mathfrak g} $ be the
semi-direct product of $ {\mathfrak g}' $ and the ideal $ {\mathfrak
H} $, where the action of $ {\mathfrak g}' $ on $ V_{i} $, $
V_{i}^{*} $ is already defined and the action of $ {\mathfrak g}' $
on $ {\mathbb C}^{2r} $ is trivial. Then $ {\mathfrak g} $ is a
contragredient Lie superalgebra, its Cartan matrix $ A $ has size $
n+r $ and can be obtained from $ A' $ in the following manner
\begin{equation}
a_{ij}=a'_{ij}\text{, if }i,j\leq n\text{, }a_{ij}=\lambda_{i-n}\left(h'_{j}\right)\text{ if }i>n,j\leq n\text{, }a_{ij}=0\text{ if }i\leq n,j>n.
\notag\end{equation}
\end{corollary}

Let $ {\mathfrak g} $ be a contragredient Lie superalgebra, by Corollary~%
\ref{cor221} $ {\mathfrak g}={\mathfrak g}'\rtimes {\mathfrak H} $. One can apply Corollary~\ref{cor221} to $ {\mathfrak g}' $ again and obtain
$ {\mathfrak g}'={\mathfrak g}''\rtimes {\mathfrak H}' $, then repeat for $ {\mathfrak g}'' $ and get $ {\mathfrak g}''={\mathfrak g}'''\rtimes {\mathfrak H}'' $ e.t.c.

\begin{lemma} \label{lm81}\myLabel{lm81}\relax  If $ {\mathfrak g} $ is a contragredient algebra of finite growth, then
$ {\mathfrak H}' $ is finite-dimensional, $ \left({\mathfrak g}''\right)^{\pm} $ are purely odd,
$ {\mathfrak H}'''=0 $ and hence $ {\mathfrak g}''''={\mathfrak g}''' $.

\end{lemma}

\begin{proof} Let $ i\in I'' $, then $ h_{i} $ is a central element of $ {\mathfrak H}' $. There exists
$ j\in I' $ such that $ \alpha_{j}\left(h_{i}\right)\not=0 $. The corresponding irreducible component $ {\mathfrak g}_{j}^{-} $ has
a non-zero central charge. A highest weight module over a Heisenberg
Lie superalgebra with a non-zero central charge has finite growth only
if the algebra is finite dimensional. Thus, $ {\mathfrak H}'_{i} $ is finite-dimensional
for each $ i\in I'' $, which implies that $ {\mathfrak H}' $ is finite-dimensional.

Repeat now this argument for $ {\mathfrak g}'' $. Let $ i\in I''' $, then there exists $ j\in I'' $
such that $ \alpha_{j}\left(h_{i}\right)\not=0 $. Then the corresponding irreducible component $ \left({\mathfrak g}'_{j}\right)^{-} $
is finite-dimensional. A highest weight module with non-zero central
charge over a Heisenberg algebra is finite-dimensional if and only if the
quotient of the algebra by the center is purely odd. Hence $ {\mathfrak H}''_{i}/{\mathbb C}h_{i} $ is
purely odd, which implies $ \left({\mathfrak g}''\right)^{\pm} $ are purely odd.

Finally, applying this argument to $ {\mathfrak g}''' $ and using the fact that there
is no purely odd highest weight modules with non-zero central charge over
a Heisenberg algebra, we obtain the last statement of Lemma.\end{proof}

\begin{theorem} \label{th122}\myLabel{th122}\relax  Let $ {\mathfrak g} $ be an affine Lie algebra, $ V $ be a non-trivial
irreducible highest weight $ {\mathfrak g} $-module. Then $ V $
has infinite growth.

\end{theorem}

\begin{proof} Let $ \lambda $ be a highest weight of $ V $, $ c $ be the central element
of $ {\mathfrak g} $. Assume first, that
$\lambda\left(c\right)\not=0 $. Let $ {\mathcal H} $ be an infinite
Heisenberg subalgebra of $ {\mathfrak g} $ generated by all
imaginary roots. If $ v $ is a highest vector of $ V $ and $ M $
be the $ {\mathcal H} $-submodule of $ V $ generated by $ v $.
Then $ M $ is the irreducible highest weight $ {\mathcal H}
$-module with non-zero central charge. If $ {\mathcal H}^{-} $ is
the subalgebra of $ {\mathcal H} $ generated by all negative
imaginary roots, then $ M $ is a free $ U\left({\mathcal
H}^{-}\right) $-module, and therefore $ M $ has infinite growth.
Hence $ V $ has infinite growth.

Now assume that $ \lambda\left(c\right)=0 $. Then there exists a
simple root $ \alpha $ such that $
\lambda\left(h_{\alpha}\right)\notin{\mathbb Z}_{\geq0} $. Then $
X_{-\alpha}^{n}v\not=0 $ for any $ v $. Let $ \delta $ be a
positive imaginary root. One can see easily from the the general
description of affine algebras, that there exist $t_{i}\in{\mathfrak
  g}_{i\delta} $,
and $ f_{i}\in{\mathfrak
g}_{-\alpha+i\delta} $. such that $
\left[t_{i},f_{j}\right]=f_{i+j} $ and $
\left[f_{i},f_{j}\right]=0 $. Since $ f_{i}\in\Delta^{+} $ for $
i>0 $, then $ f_{i}v=0 $ for all $ i>0 $. Let $ {\mathfrak a} $ be
the abelian subalgebra of $ {\mathfrak g} $ spanned by $ f_{i} $
for all $ i\leq0 $. We claim that $ U\left({\mathfrak a}\right)v $
is free over $ U\left({\mathfrak a}\right) $. Indeed, assume that
\begin{equation}
\sum_{i_{1}<i_{2}<\dots <i_{k}}c_{i_{1},\dots
,i_{k}}\left(f_{-i_{1}}\right)^{p_{1}}\dots
\left(f_{-i_{k}}\right)^{p_{k}}v=0. \notag\end{equation} Choose
the maximal monomial
\begin{equation}
\left(f_{-i_{1}}\right)^{p_{1}}\dots
\left(f_{-i_{k}}\right)^{p_{k}}v \notag\end{equation} (in
lexicographical order) which appears with non-zero coefficient in
this relation and apply $ t_{i_{k}} $ to this relation. Then the
result is a similar non-zero relation of smaller degree in the
principal grading, since for any monomial
\begin{equation}
t_{i_{k}}\left(f_{-j_{1}}\right)^{q_{1}}\dots
\left(f_{-j_{l}}\right)^{q_{l}}\left(f_{-i_{k}}\right)^{s}v=sf_{0}\left(f_{-j_{1}}\right)^{q_{1}}\dots
\left(f_{-j_{l}}\right)^{q_{l}}\left(f_{-i_{k}}\right)^{s-1}v
\notag\end{equation} if $ j_{1}<\dots <j_{l}<i_{k} $. Hence the
claim follows by induction on degree and the fact that $
f_{0}^{n}v\not=0 $ for any $ n $.

By the argument similar to the proof of Lemma~\ref{lm24}, $
U\left({\mathfrak a}\right)v $ has infinite growth, hence $ V $
has infinite growth.\end{proof}

We say that a Cartan matrix $ A $ has {\em quasi-finite type\/} if
$ {\mathfrak g}\left(A\right) $ is a direct sum of finite-dimensional Lie
superalgebras with indecomposable Cartan matrices and several copies of
$ \widehat{D}\left(2,1;0\right) $.

\begin{lemma} \label{lm82}\myLabel{lm82}\relax  If $ {\mathfrak g}\left(A\right) $ is a contragredient Lie superalgebra of finite
growth with indecomposable non-elemental $ A $, then $ A''' $ is of
quasi-finite type.
\end{lemma}

The proof of this Lemma follows immediately from the following

\begin{lemma} \label{lm83}\myLabel{lm83}\relax  Let $ {\mathfrak g} $ be a contragredient Lie superalgebra of
finite-growth with a elemental indecomposable Cartan matrix and
suppose there exists a non-trivial highest weight $ {\mathfrak g}
$-module of finite growth. Then $ {\mathfrak g} $ is either
finite-dimensional or $ \widehat{D}\left(2,1,0\right) $.

\end{lemma}

\begin{proof} Assume that $ {\mathfrak g} $ is infinite-dimensional. If its Cartan
matrix is symmetrizable, then $ {\mathfrak g} $ is affine superalgebra and its even part
$ {\mathfrak g}_{0} $ is a subquotient of a direct sum of affine Lie algebras. If $ V $ is a
non-trivial $ {\mathfrak g} $-module of finite-growth, then it contains a non-trivial
irreducible $ {\mathfrak g}_{0} $-subquotient $ V_{0} $. However, it is impossible by Theorem~%
\ref{th122}. In the same manner, one can argue that $ {\mathfrak g} $ is not $ q\left(n\right)^{\left(2\right)} $, since
its even part $ {\mathfrak g}_{0} $ contains the loop algebra of $ sl\left(n\right) $. Finally,
consider the case when $ {\mathfrak g}=S\left(1,2,a\right) $. Let $ V_{\lambda} $ be an irreducible highest
weight module of finite growth with highest weight $ \lambda $. Note
that $ \mathfrak {g}_0 $ contains a subalgebra
isomorphic to $ sl\left(2\right)^{\left(1\right)} $ whose simple roots
are $\alpha _1+ \alpha _2$ and $\alpha_3$. By Theorem~\ref{th122}, $ \lambda\left(h_{3}\right)=0 $ and
\begin{equation}
\lambda\left(h_{\alpha_{1}+\alpha_{2}}\right)=\lambda\left(h_{1}\right)-\lambda\left(h_{2}\right)=0.
\label{equ83}\end{equation}\myLabel{equ83,}\relax
Assume that $ \lambda\left(h_{1}\right)\not=0 $. Then $ \lambda-\alpha_{1} $ be the highest weight of $ V_{\lambda} $ with respect to
the base $ -\alpha_{1} $, $ \alpha_{1}+\alpha_{2} $, $ \alpha_{1}+\alpha_{3} $, obtained from initial base by the odd
reflection $ r_{1} $. Then we have
\begin{equation}
\lambda-\alpha_{1}\left(h_{1}-h_{2}\right)=\lambda\left(h_{1}\right)-\lambda\left(h_{2}\right)+1=0,
\notag\end{equation}
which obviously contradicts~\eqref{equ83}.

Thus, we showed that $ {\mathfrak g} $ is finite-dimensional or $ \widehat{D}\left(2,1;0\right) $. The latter
case is possible, since the superalgebra has a finite-dimensional
quotient isomorphic to $ sl\left(2\right) $, hence one can consider any highest weight
$ sl\left(2\right) $-module and induce $ \widehat{D}\left(2,1;\alpha\right) $ action on it assuming that the ideal
acts trivially.\end{proof}

\begin{theorem} \label{th84}\myLabel{th84}\relax  The contragredient Lie superalgebra $ {\mathfrak g}\left(A\right) $ with a
non-elemental indecomposable Cartan matrix has finite growth if and only if
the following conditions hold
\begin{enumerate}
\item
$ A''' $ is of quasi-finite type;
\item
$ p\left(i\right)=1 $ for every $ i\in I''' $;
\item
for every $ i\in I''' $, $ j\in I-\left(I'\cup I''\cup I'''\right) $, $ a_{ji}\not=0 $ implies that $ \alpha_{j} $ is either
a simple root of a purely even block or is the left simple root of the diagram in table 8.1 in $ D\left(2,1;0\right) $ or
$ \widehat{D}\left(2,1;0\right) $ block;
\item
for any $ i\in I'' $, the restriction of $ -\alpha_{i} $ on the Cartan subalgebra of
$ {\mathfrak g}''' $ is an integral dominant weight (a highest weight of
finite-dimensional irreducible $ {\mathfrak g}''' $-module);
\item
if $ i\in I $, $ j $ is an index from $ \widehat{D}\left(2,1;0\right) $-block of $ A''' $ and $ a_{ji}\not=0 $, then
$ \alpha_{j} $ is the left simple root in the diagram in table 8.1.
\end{enumerate}
\end{theorem}

The proof of this theorem follows from Lemmas~\ref{lm81},
~\ref{lm82} and Corollary~\ref{cor221} and we will leave it to the
reader.

\end{document}